\documentclass[12pt,a4paper]{amsart}
\usepackage{geometry}
\usepackage{amsmath}
\usepackage{amssymb}
\usepackage{amsfonts}
\usepackage{tikz-cd}
\usepackage{tikz}
\usepackage{circuitikz}
\usetikzlibrary{ decorations.markings}
\usepackage{enumitem}
\usepackage{graphicx}

\newtheorem{theorem}{Theorem}[section]

\newtheorem{corollary}[theorem]{Corollary}

\newtheorem{definition}[theorem]{Definition}
\newtheorem{example}[theorem]{Example}

\newtheorem{lemma}[theorem]{Lemma}
\newtheorem{Fact}[theorem]{Fact}

\newtheorem{ques}[theorem]{Question}
\newtheorem{proposition}[theorem]{Proposition}
\newtheorem{remark}[theorem]{Remark}

    \newcommand{\Art}{{\mathrm{Art}}} 

\begin{document}

\title{Splittings and  poly-freeness of triangle Artin groups}
\author{Xiaolei Wu}
\address{Shanghai Center for Mathematical Sciences, Jiangwan Campus, Fudan University, No.2005 Songhu Road, Shanghai, 200438, P.R. China}
\email{xiaoleiwu@fudan.edu.cn}

\author{Shengkui Ye}
\address{NYU Shanghai, No.567 Yangsi West 
  Rd, Pudong New Area, Shanghai, 200124, P.R. China \\
NYU-ECNU Institute of Mathematical Sciences at NYU Shanghai, 3663 Zhongshan Road North, Shanghai, 200062, China}
\email{sy55@nyu.edu}


\maketitle

\begin{abstract}
We prove that the triangle Artin group $\mathrm{Art}_{23M}$ splits as a
graph of free groups if and only if $M$ is greater than $5$ and even.  This answers two questions of 
Jankiewicz \cite[Question 2.2, Question 2.3]{Jan21} in the negative. Combined with the results of Squier and Jankiewicz, this completely determines when a triangle Artin group splits as a graph of free groups. Furthermore, we prove that
 the triangle Artin groups are virtually poly-free when the labels are not
of the form $(2,3, 2k+1)$ with $k\geq 3$. This partially answers a question of Bestvina \cite{Be99}.
\end{abstract}

\section{Introduction}
The study of groups given by presentations has a long history. Artin groups
 provide a particularly interesting class of such groups that has been under intensive research for decades.  Although there are already tremendous
studies, the structures of Artin groups are still full of mystery.  In this
article, we focus on the triangle Artin groups which can be defined by the following presentation 
\begin{equation*}
\mathrm{Art}_{MNP}=\langle a,b,c\mid
\{a,b\}_{M}=\{b,a\}_{M},\{b,c\}_{N}=\{c,b\}_{N},\{c,a\}_{P}=\{a,c\}_{P}%
\rangle,
\end{equation*}%
where $\{a,b\}_{M}$ denotes the alternating product $aba\cdots $ of length $M
$ starting with $a$, $2\leq M,N,P <\infty$.
Squier \cite{Sq87} showed that the Euclidean triangle Artin groups, i.e.~\textrm{Art}$_{236}$, \textrm{Art}$_{244}$ and \textrm{Art}$_{333}$, split
as amalgamated products or an HNN extension of finite-rank free groups along
finite index subgroups. This was generalized by Jankiewicz \cite{Jan22,Jan21}
to  triangle Artin groups $\mathrm{Art}_{MNP}$ with $M\leq N\leq P$
satisfying either $M>2$, or $N>3.$ Thus it is a natural question whether such splitting holds for all triangle Artin groups. In fact, the following two questions were asked by
Jankiewicz \cite[Question 2.2, Question 2.3]{Jan21} :

\begin{ques}\label{ques-split-tri-art}
Does the Artin group $\mathrm{Art}_{23M}$ where $M\geq 7$ splits as a graph
of finite rank free groups?
\end{ques}


\begin{ques}\label{ques-split-2dim-art}
Do all 2-dimensional Artin groups split as a graph of finite rank free
groups?
\end{ques}

Jankiewicz further conjectured in \cite{Jan21} that the answer to Question \ref{ques-split-tri-art} should be positive.  In this article, we give negative answers to both questions. Let $F_n$ be the free group of rank $n$, our main theorem can be stated as the following:

\begin{theorem}
\label{th1}When $M$ is odd, the Artin group $\mathrm{Art}_{23M}$ can not
split as a nontrivial graph of free groups. When $M>4$ is even, the Artin
group $\mathrm{Art}_{23M}$ is isomorphic to $F_{3}\ast _{F_{7}}F_{4},$ an
amalgamated product the free groups $F_{3}$ and $F_{4}$ over a subgroup $%
F_{7}. $
\end{theorem}
\begin{remark}
   When $M\leq 5,$ the Artin group $\mathrm{Art}_{23M}$ is of finite type and of geometric dimension 3. Thus it can not split as a graph of free groups since a graph of free groups has geometric dimension $2$ (see also 
\cite[Corollary 2.11]{Jan21} for a different argument). Combined with results in \cite{Sq87,Jan22, Jan21}, Theorem \ref{th1} implies that a triangle Artin group $A_{MNP}$ splits as a graph of free groups if and only if it is neither of finite type, nor of the form $A_{23M}$ where $M$ is odd. 
\end{remark}

Related to the graphs of free group structure, one also has the following question of Bestvina \cite{Be99}.

\begin{ques}\label{ques-vpf}
  Is every Artin group (of finite type)  virtually
poly-free?  
\end{ques}

Recall that a group $G$ is poly-free if there is a finite subnormal sequence%
\begin{equation*}
1<G_{1}\trianglelefteq G_{2}\trianglelefteq \cdots \trianglelefteq G_{n}=G
\end{equation*}%
such that the successive quotient $G_{i+1}/G_{i}$ is free for each $i$. Note that a poly-free group is locally indicable, i.e.~every finitely generated subgroup surjects to the infinite cyclic group $\mathbb{Z}$. And a group $G$ has a virtual property $P$ if there is a finite-index subgroup $H \leq G$ has property $P$.    Although some progress has been made on Question \ref{ques-vpf} \cite{BG21,BMP19,Wu22}, it is still open widely for
 Artin groups in general. Restricted to the triangle Artin groups, they are known to be virtually poly-free, if they are of finite type, Euclidean type (see the proof of Theorem \ref{th3} on p.\pageref{proof:thm3} for a detailed discussion), or  all labels are even \cite{BG21}.  Our next theorem extends this significantly.

\begin{theorem}
\label{th3}Suppose that $ M\leq N \leq P$ and $(M,N,P) \neq (2,3, 2k+1)$ for some integer $k\geq 3$. Then the Artin
group $\mathrm{Art}_{MNP}$ is virtually poly-free.
\end{theorem}

Note that  the  word ``virtual" in Theorem \ref{th3} cannot be dropped by the following theorem.

\begin{theorem}
\label{th2}Assume that $M,N,P$ are pairwise coprime integers. Then the commutator subgroup
of the Artin group $\mathrm{Art}_{MNP}$ is perfect. In particular, the Artin
group $\mathrm{Art}_{MNP}$ is not poly-free.
\end{theorem}

\begin{remark}
    In contrast to our theorem, Blasco-Garcia proved in \cite{BG21} that triangle Artin groups with even labeling indeed are poly-free.
\end{remark}

  The
splitting of Artin groups $\mathrm{Art}_{23M}$ (for even $M$) in Theorem \ref{th1} is proved in Section \ref{sec:split-23even} which is based on a presentation studied by Hanham \cite{Han02}. The proof of non-splitting of Artin
groups $\mathrm{Art}_{23M}$ (for odd $M$) in Theorem \ref{th1} is based on a
study of group action on trees. In fact, we prove in Section \ref{sec:split-23odd} that any isometric action of such an
Artin group on a simplicial tree must have a fixed point or an invariant
geodesic line or is close to an action on a Bass--Serre tree with finite vertex stabilizers (see Theorem \ref{thm-tree-action-clf}).  In order to prove Theorem \ref{th3},
we establish a criterion of poly-freeness for multiple HNN extensions. We
prove that an algebraically clean graph of free groups (i.e.~each edge subgroup
is a free factor of the corresponding vertex subgroup) is poly-free. Since  the Artin group $\mathrm{Art}_{23M}$ for $M$ is an odd
integer at least 7 does not split as a graph of free groups, our method will not work for these groups. In particular, it is still an open problem whether the Artin
group $\mathrm{Art}_{2,3,2k+1},k>2,$ is virtually poly-free.
Theorem \ref{th2} is proved  in the last section by writing down the presentation of the commutator subgroup
using the Reidemeister-Schreier method.  

\subsection*{Acknowledgements.} Wu is currently a member of LMNS and supported by  a starter grant at Fudan University. He thanks Georges Neaime for helpful discussions regarding poly-freeness of Artin groups, and  Jingyin Huang for pointing out the paper \cite{Jan22}. Ye is supported by NSFC (No. 11971389). After we finished the paper, we learned from Kasia Jankiewicz that she and Kevin Schreve \cite{JS23} have independently proved that algebraically clean graph of groups (here they assume that the graph is finite and the free group is of finite rank) are normally poly-free. In particular, they also have a proof for Corollary \ref{cor:vnpf}, which covers the poly-freeness of many triangle Artin groups. She further informed us that her student Greyson Meyer has independently obtained a proof of Theorem \ref{23even}, i.e.~Artin groups of the form $(2,3,2n)$, where $n\geq 3$, admit graph-of-free-group splittings. We thank the referee for suggesting numerous corrections and clarifications to an earlier draft of this paper.

\bigskip

\section{ Basic notations and facts}
We introduce some basic notations and facts that we will need in the next sections.
\subsection{Poly-free groups}

Recall  that a group $G$ is poly-free if there is a finite
subnormal sequence%
\begin{equation*}
1<G_{1}\trianglelefteq G_{2}\trianglelefteq \cdots \trianglelefteq G_{n}=G
\end{equation*}%
such that the quotient $G_{i}/G_{i+1}$ is free for each $i$. Here the rank of $%
G_{i}/G_{i+1}$ is not required to be finite. When each quotient $%
G_{i}/G_{i+1}$ is cyclic (or infinite cyclic), the group $G$ is called polycyclic (resp.
poly-$\mathbb{Z}$). The group $G$ is called normally poly-free if
additionally each $G_{i}$ is normal in $G$. The minimal such $n$ is called
the poly-free (resp. normally poly-free) length of $G$. Note that not every
poly-free group is normally poly-free.

\begin{example}
Let $n\geq 2$. For an irreducible invertible $n\times n$ integer matrix $A$, the semi-direct
product $\mathbb{Z}^{n}\rtimes _{A}\mathbb{Z}$ is poly-free, but not normally
poly-free.
\end{example}

\begin{proof}
Since $\mathbb{Z}^{n}$ is poly-free, the group $\mathbb{Z}^{n}\rtimes \mathbb{%
Z}$ is poly-free. But any non-trivial normal subgroup $H$ containing a direct summand $\mathbb{Z}^{n}$ is
the whole group $\mathbb{Z}^{n}$ which is not a free group when $n\geq 2$. This implies that $\mathbb{Z}^{n}\rtimes 
\mathbb{Z}$ does not have a subnormal series of normal subgroups with successive free quotients. Therefore, it cannot be
normally poly-free.
\end{proof}

Poly-free groups have many nice properties. For example, they are

\begin{enumerate}

\item torsion-free;

\item locally indicable;

\item having finite asymptotic dimension (one can prove this via induction
using \cite[Theorem 2.3]{DS06});

\item satisfying the Baum--Connes Conjecture with coefficients \cite[Remark 2]{Wu22}, and

\item satisfying the Farrell--Jones Conjecture if they are normally
poly-free \cite[Theorem A]{BKW21}.
\end{enumerate}

It is clear that the class of poly-free groups is closed under taking
subgroups and extensions. The following is also a
folklore (cf. Bestvina \cite{Be99}). Since we cannot find a proof, we
provide a detailed argument here.

\begin{lemma}
A poly-free group satisfies the Tits alternative, i.e.~any nontrivial
subgroup either is polycyclic or contains a nonabelian free subgroup.
\end{lemma}

\begin{proof}
Let $G$ be a poly-free group with a subnormal sequence%
\begin{equation*}
1<G_{1}\trianglelefteq G_{2}\trianglelefteq \cdots \trianglelefteq G_{n}=G
\end{equation*}%
such that the quotient $G_{i}/G_{i+1}$ is free for each $i.$ Suppose that $%
1\neq H<G$ is a subgroup. Consider the image $H_{1}=\mathrm{Im}%
(H\hookrightarrow G\rightarrow G/G_{n-1}),$ where the first homomorphism is
the inclusion and the second homomorphism is the natural quotient. Since $%
G/G_{n-1}$ is free, the image $H_{1}$ is free. If $H_{1}$ is non-abelian, we
choose a section of the epimorphism $H\rightarrow H_{1}$ to get a
non-abelian free subgroup of $H$. If $H_{1}$ is (nontrivial or trivial)
abelian, we consider the $\ker (H\rightarrow H_{1}),$ which is a subgroup of 
$G_{n-1}.$ Repeat the argument for $\ker (H\rightarrow H_{1})<G_{n-1}$,
instead of $H<G.$ After at most $n$ steps, the argument is stopped. If there
exists a non-abelian image in $G_{i}/G_{i-1}$ for some $i,$ the subgroup $H$
contains a non-abelian free subgroup. Otherwise, the subgroup $H$ is
polycyclic.
\end{proof}

We will also need the following fact.

\begin{lemma}
\label{2.6}
\cite[Lemma 2.6]{Wu22} Let $G=A\ast _{C}B$ be an amalgamated product of two
groups $A,B$ along subgroup $C.$ Suppose that $f:G=A\ast _{C}B\rightarrow Q$
is a group homomorphism such that $ f|_{C}$ is injective. If $\ker
f|_{A},\ker f|_{B}$ are free, then $\ker f$ is free. In particular, when $%
\ker f|_{A},\ker f|_{B}$ are free and $Q$ is (resp. normally) poly-free, the
group $G$ is (resp. normally) poly-free.
\end{lemma}

\subsection{Graphs, immersions and oppressive sets}
In this paper, a graph is a 1-dimensional CW-complex. A map $f:X_{1}\rightarrow
X_{2}$ between two graphs $X_{1},X_{2}$ is combinatorial if the image $f(v)$
of each vertex $v\in X_{1}$ is a vertex of $X_{2}$, and each open edge $%
[v_{1},v_{2}]$ with endpoints $v_{1},v_{2}$ in $X_{1}$ is mapped
homeomorphically onto an edge with endpoints $f(v_{1}),f(v_{2})$ in $X_{2}$. A
combinatorial map $f:X_{1}\rightarrow X_{2}$ is a combinatorial immersion,
if $f$ is locally injective, i.e.~for every vertex  $v\in X_{1}$ and
oriented edges $e_{1},e_{2}$  with terminal vertex $v$ such that $%
f(e_{1})=f(e_{2})$, we have $e_{1}=e_{2}$. 

\begin{lemma}
\label{stallings}\cite[Proposition 5.3]{sta} A combinatorial immersion $%
f:X_{1}\rightarrow X_{2}$ induces an injective fundamental group
homomorphism 
\begin{equation*}
f_{\ast }:\pi _{1}(X_{1},v)\hookrightarrow \pi _{1}(X_{1},f(v))
\end{equation*}%
for any vertex $v\in X_{1}.$
\end{lemma}

A combinatorial map $f:X_{1}\rightarrow X_{2}$ between (oriented) graphs is
a Stallings' folding if 

\begin{itemize}
\item[(i)] there exist distinct (closed) edges $e_{1},e_{2}$ of $X_{1}$
starting from a common vertex such that $X_{2}=X_{1}/e_{1}\sim e_2$%
, i.e.~$X_{2}$ is obtained from $X_{1}$ by identifying $e_{1}$ with $e_{2},$
and

\item[(ii)] the natural quotient map $X_{1}\rightarrow X_{1}/e_{1}\sim e_{2}$ is $f.$
\end{itemize}
When $X_{1}$ is a finite graph and  $f:X_{1}\rightarrow X_{2}$ is a
combinatorial map of graphs, there is a finite sequence of combinatorial
maps of graphs%
\begin{equation*}
X_{1}=\Gamma _{0}\rightarrow \Gamma _{1}\rightarrow \cdots \rightarrow
\Gamma _{n}\rightarrow X_{2}
\end{equation*}
such that $f$ is the composition of the immersion $\Gamma _{n}\rightarrow
X_{2}$ and a sequence of surjective graph maps $\Gamma _{i}\rightarrow
\Gamma _{i+1}$ (Stallings' foldings), for $i=0,1,2,...,n-1$. The sequence
of foldings is not unique, but the final immersion is unique (cf. \cite[p. 555]{sta}).

Let $(X,x_{0}),(Y,y_{0})$ be based graphs and $\rho :(Y,y_{0})\rightarrow
(X,x_{0})$ be a combinatorial immersion. A path in a graph is called simple if it is an embedding. The oppressive set $A_{\rho
}\subseteq \pi _{1}(X,x_{0})$ consists of all $g\in \pi _{1}(X,x_{0})$
representing by a cycle $\gamma $ in $X$ such that $\gamma =\rho (\mu
_{1})\cdot \rho (\mu _{2}),$ a concatenation, where $\mu _{1}$ is a
non-trivial simple non-closed path in $Y$ going from $y_{0}$ to some vertex $%
y_{1},$ and $\mu _{2}$ is either trivial or a simple non-closed path in $Y$
going from some vertex $y_{2}$ to $y_{0},$ with $y_{1}\neq y_{2}\neq y_{0}.$
Note that $\rho (y_{1})=\rho (y_{2}).$ If $f:H\rightarrow G$ is an injection
of free groups induced by an explicit graph immersion, we also denote by $A_{f}$ the
oppressive set.

\begin{lemma}
\label{2.4}\cite[Proposition 2.4 (4)]{Jan22} Suppose that the combinatorial immersion $\rho
:(Y,y_{0})\rightarrow (X,x_{0})$ induces an injection $\rho :\pi
_{1}(Y,y_{0})\hookrightarrow \pi _{1}(X,x_{0}).$ Let $\phi :\pi
_{1}(X,x_{0})\rightarrow \bar{G}$ be a group homomorphism that separates $%
\pi _{1}(Y,y_{0})$ from $A_{\rho }$, i.e.~$\phi(\pi _{1}(Y,y_{0})) \cap \phi(A_{\rho }) = \emptyset $. Then $\pi _{1}(Y,y_{0})\cap \ker \phi $
is a free factor in $\ker \phi .$
\end{lemma}

We will need the following lemma from \cite[Lemma 2.5]{Jan22}.

\begin{lemma}
\label{lemma-imy-sep}
Let $\rho: (Y,y_0) \rightarrow (X,x_0)$ be a combinatorial immersion of graphs with $x_0$ the unique vertex of $X$.  Let $Y_{e}, X_{e}$ be 2-complexes with 1-skeletons $Y_{e}^{(1)}=Y, X_{e}^{(1)}=X$ and let $\rho_e: Y_e \rightarrow X_e$ a map extending $\rho$.  Let  $\phi: \pi_1(X,x_0) \rightarrow \pi_1(X_{e},x_0)$ be the natural quotient and suppose that $\phi(\pi_1(Y,y_0)) = \rho_{e*} (\pi_1(Y_{e},y_0))$. If the lift to the universal covers  $\widetilde{\rho_e}:  \widetilde{Y_e} \rightarrow \widetilde{X_e}$ of $\rho_e$ is an embedding, then $\phi$ separates $\pi_1(Y,y_0)$ from $A_{\rho}$.
\end{lemma}

\subsection{Reidemeister--Schreier rewriting procedure}

Let $G=\langle a_{i},i\in I|r_{i},i\in J\rangle $ be a presentation of a
group, where $I,J$ are two index sets.  For a subgroup $H\leq G$,  a system $R$
of words in the generators $a_{i},i\in I$, is called a Schreier system for $G
$ modulo $H$ if 

\begin{itemize}
\item[(i)] every right coset of $H$ in $G$ contains exactly one word of $R$
(i.e.~$R$ forms a system of right coset representatives); 

\item[(ii)] for each word in $R$ any initial segment is also in $R$ (i.e.~initial segments of right coset representatives in $R$ are again right coset
representatives). 

\item[(iii)] the empty word $\emptyset \in R.$
\end{itemize}

For any $g\in G,$ let $\bar{g}$ be the unique element in $R$ such that $Hg=H%
\bar{g}.$ For each $K\in R,a=a_{i},$ let $s_{K,a}=Ka\overline{Ka}^{-1}\in H.$
For more details on the Schreier system, see \cite[Section 2.3]{mks}. The
following Reidemeister theorem gives a presentation of the subgroup $H$.

\begin{theorem}
\label{schreier}
\cite[Corollary 2.7.2 + Theorem 2.8]{mks} The subgroup $H$ has a
presentation%
\begin{eqnarray*}
\langle s_{K,a},K &\in &R,a\in \{a_{i},i\in I\}\mid s_{K,a}=1,K\in R,a\in
\{a_{i},i\in I\},\text{if }Ka\equiv \overline{Ka}, \\
\tau (Kr_{i}K^{-1})& =&1,i\in J,K\in R\rangle ,
\end{eqnarray*}%
where $Ka\equiv \overline{Ka}$ means that the two words are equivalent in the
free group $F(\{a_{i},i\in I\})$, and $\tau \ $is the
Reidemeister--Schreier rewriting function defined as follows%
\begin{eqnarray*}
\tau  &:&F(\{a_{i},i\in I\})\rightarrow F(\{s_{K,a},K\in R,a\in \{a_{i},i\in
I\}\}), \\
a_{i_{1}}^{\varepsilon _{1}}a_{i_{2}}^{\varepsilon _{2}}\cdots
a_{i_{m}}^{\varepsilon _{m}} &\mapsto &s_{K_{i_{1}},a_{i_{1}}}^{\varepsilon
_{1}}s_{K_{i_{2}},a_{i_{2}}}^{\varepsilon _{2}}\cdots
s_{K_{i_{m}},a_{i_{m}}}^{\varepsilon _{m}}
\end{eqnarray*}
with%
\begin{equation*}
K_{i_{j}}=\left\{ 
\begin{array}{c}
\overline{a_{i_{1}}^{\varepsilon _{1}}a_{i_{2}}^{\varepsilon _{2}}\cdots
a_{i_{j-1}}^{\varepsilon _{j-1}}},\text{ if }\varepsilon
_{j}=1, \\ 
\overline{a_{i_{1}}^{\varepsilon _{1}}a_{i_{2}}^{\varepsilon _{2}}\cdots
a_{i_{j}}^{\varepsilon _{j}}},\text{ if }\varepsilon _{j}=-1.%
\end{array}%
\right. 
\end{equation*}
\end{theorem}

Note that $\tau (g_{1}g_{2})=\tau (g_{1})\tau (g_{2})$ for any $%
g_{1},g_{2}\in H$ and $\tau (Ka\overline{Ka}^{-1})=s_{K,a}$ (cf. \cite[Theorem 2.6 (7) and Corollary 2.7.2]{mks}).

\begin{lemma}
\label{shift}Suppose that $H$ is normal in a group $G,$ and $R$ is a
Schreier system for $G$ modulo $H$. For any $r,K\in R,a\in \{a_{i},i\in I\},$
we have $rs_{K,a}r^{-1}=s_{rK,a}$ as long as $rK,r\overline{Ka}\in R.$
\end{lemma}
\begin{remark}
    Here the equation $rs_{K,a}r^{-1}=s_{rK,a}$ means the two sides are equal as elements in $G$.
\end{remark}
\begin{proof}
Note that $HKa=H\overline{Ka},$ and $rHKa=H(rKa)=rH\overline{Ka}=Hr\overline{%
Ka}.$ If $r\overline{Ka}\in R,$ we have $r\overline{Ka}=\overline{rKa}$.
Therefore, $s_{rK,a}=rKa\overline{rKa}^{-1}=rKa\overline{Ka}%
^{-1}r^{-1}=rs_{K,a}r^{-1}.$
\end{proof}

\section{Splitting of the Artin group $(2,3,2m)$}\label{sec:split-23even}
In this section, we will prove that the triangle Artin group $\mathrm{Art}%
_{2,3,2m}$ splits as an amalgamated product of the form $F_{3}\ast _{F_{7}}F_{4}$ when $m\geq 3$.  The proof is based on a
presentation of $\mathrm{Art}_{2,3,2m}$ obtained by Hanham \cite[p.41-42]%
{Han02}, who actually proves that the corresponding presentation complex is 
CAT(0).

Let $$\mathrm{Art}_{2,3,2m}=\langle
a,b,c:ac=ca,bcb=cbc,(ab)^{m}=(ba)^{m}\rangle $$ be the Artin group of type $%
(2,3,2m),m\geq 3.$ Let 
\begin{equation*}
H=\langle b,c,x,y,\alpha ,\delta :\alpha =xc,\alpha =bx,y=bc,yb=cy,\delta
b=c\delta ,\delta =\alpha x^{m-2}\alpha \rangle .
\end{equation*}

\begin{lemma}
 \cite[p.41-42]{Han02} The group $\mathrm{Art}_{2,3,2m}$ is isomorphic to $H$.
\end{lemma}

\begin{proof}
Define $\phi :\mathrm{Art}_{2,3,2m}\rightarrow H$ by 
\begin{equation*}
\phi (a)=b^{-1}c^{-1}xc,\phi (b)=b,\phi (c)=c
\end{equation*}%
and $\psi :H\rightarrow \mathrm{Art}_{2,3,2m}$ by%
\begin{eqnarray*}
\psi (b) &=&b,\psi (c)=c,\psi (x)=c(ba)c^{-1}, \\
\psi (y) &=&bc,\psi (\alpha )=bc(ba)c^{-1},\psi (\delta )=bc(ba)^{m}.
\end{eqnarray*}%
Note that%
\begin{eqnarray*}
\phi (ca) &=&cb^{-1}c^{-1}xc=b^{-1}c^{-1}cbcb^{-1}c^{-1}xc \\
&=&b^{-1}c^{-1}(bcb)b^{-1}c^{-1}xc=b^{-1}c^{-1}bxc=b^{-1}c^{-1}xcc=\phi (ac)
\end{eqnarray*}%
since $bx=xc,$ and%
\begin{eqnarray*}
\phi (bcb) &=&bcb=yb=cy=cbc=\phi (cbc), \\
\phi ((ab)^{m})
&=&(b^{-1}c^{-1}xcb)^{m}=b^{-1}c^{-1}x^{m}cb=b^{-1}c^{-1}(b^{-1}\delta
c^{-1})cb \\
&=&b^{-1}c^{-1}b^{-1}\delta b=c^{-1}b^{-1}c^{-1}(c\delta
)=c^{-1}b^{-1}\delta , \\
\phi ((ba)^{m}) &=&(bb^{-1}c^{-1}xc)^{m}=c^{-1}x^{m}c=c^{-1}b^{-1}\delta
c^{-1}c=c^{-1}b^{-1}\delta .
\end{eqnarray*}%
Furthermore, we have%
\begin{eqnarray*}
\psi (\alpha ) &=&bc(ba)c^{-1}=b\psi (x), \\
\psi (\delta b) &=&bc(ba)^{m}b=bcb(ab)^{m}=cbc(ba)^{m}=\psi (c\delta ), \\
\psi (\alpha x^{m-2}\alpha ) &=&bc(ba)c^{-1}c(ba)^{m-2}c^{-1}bc(ba)c^{-1} \\
&=&bc(ba)^{m-1}c^{-1}(cbc)c^{-1}a=bc(ba)^{m}=\psi (\delta ).
\end{eqnarray*}
This checked that $\phi ,\psi $ are group homomorphisms. It is obvious that $%
\psi \circ \phi =id,\phi \circ \psi =id.$
\end{proof}

Note that $H$ has the presentation complex with a single vertex as shown in Figure \ref{fig:pre-cpx-2-3-2m} where all the edges with the same labels are identified. Let us denote this complex by $X_H.$
%

\begin{figure}[h]
\centering
\begin{tikzpicture} [line width=1pt, scale = 0.4]


 \begin{scope}[very thick,decoration={
    markings,
    mark=at position 0.5 with {\arrow{>}}}
    ] 
 
 \draw[postaction={decorate}] (-16,0) -- (8,0);
  \draw[postaction={decorate}] (-16,0) -- (-12,0);
 \draw[postaction={decorate}] (4,0) -- (8,0);
\draw[postaction={decorate}] (-15,-3) -- (-16,0);
\draw[postaction={decorate}] (4,0) -- (5, -3);
\draw[postaction={decorate}] (8,0) -- (5, -3);
\draw[postaction={decorate}] (-15,-3) -- (-19, -3);
\draw[postaction={decorate}] (-15,-3) -- (5, -3);
\draw[postaction={decorate}] (-15,-3) -- (-12, 0);
\draw[postaction={decorate}] (-16,0) -- (-19, -3);
\draw[postaction={decorate}] (-19,-3) -- (-22, -6);
 \draw[postaction={decorate}] (5,-3) -- (2, -6);
 \draw[postaction={decorate}] (-15,-3) -- (-18, -6);
\draw[postaction={decorate}] (-18,-6) -- (-22, -6);
\draw[postaction={decorate}] (-18,-6) -- (2, -6);

  \draw (12,0) -- (12,-6);
   
      

\end{scope}

  \filldraw

  (-16,0) circle (2.5pt)
  (4,0) circle (2.5pt)
  (8,0) circle (2.5pt)
  (-12,0) circle (2.5pt)

  (-19,-3)circle (2.5pt)
  (-15,-3)circle (2.5pt)
  (5,-3) circle (2.5pt)

  (-18,-6)circle (2.5pt)
  (2,-6)circle (2.5pt)
  (-22,-6) circle (2.5pt);
  
\draw (-5,-2.5)  node[text=black, scale=.7]{$\delta$};
\draw (-8,-5.5)  node[text=black, scale=.7]{$\delta$};

\draw (-3.4,.5)  node[text=black, scale=.7]{$x^{m-2}$};
\draw (-14,.5)  node[text=black, scale=.7]{$x$};

\draw (6,.5)  node[text=black, scale=.7]{$x$};

\draw (6,.5)  node[text=black, scale=.7]{$x$};

\draw (-17.8,-1.3)  node[text=black, scale=.7]{$c$};
\draw (-15.9,-1.3)  node[text=black, scale=.7]{$b$};
\draw (-13.8,-1.2)  node[text=black, scale=.7]{$\alpha$};
\draw (3.8,-1.3)  node[text=black, scale=.7]{$\alpha$};
\draw (7.2,-1.3)  node[text=black, scale=.7]{$c$};

\draw (-5,-2.5)  node[text=black, scale=.7]{$\delta$};
\draw (-17,-2.5)  node[text=black, scale=.7]{$y$};

\draw (-20.8,-4.3)  node[text=black, scale=.7]{$b$};
\draw (-15.9,-4.3)  node[text=black, scale=.7]{$c$};
\draw (4.2,-4.3)  node[text=black, scale=.7]{$b$};
  
\draw (-20,-5.5)  node[text=black, scale=.7]{$y$}; 

 \draw (12,0)  node[text=black, scale=.7]{$-$};  

  \draw (12,-1.5)  node[text=black, scale=.7]{$-$};

   \draw (12,-3)  node[text=black, scale=.7]{$-$};

    \draw (12,-4.5)  node[text=black, scale=.7]{$-$};
        \draw (12,-6)  node[text=black, scale=.7]{$-$};

     \draw (12.5,0)  node[text=black, scale=.7]{$0$};  

  \draw (12.5,-1.5)  node[text=black, scale=.7]{$\frac{1}{2}$};

   \draw (12.5,-3)  node[text=black, scale=.7]{$1$};

    \draw (12.5,-4.5)  node[text=black, scale=.7]{$\frac{1}{2}$};
        \draw (12.5,-6)  node[text=black, scale=.7]{$0$};    
 
\draw[->](8.5,-3) -> (10.5,-3);
  \draw (9.5,-2.5)  node[text=black, scale=.7]{$h$};
   
\end{tikzpicture}
\caption{The presentation complex $X_H$}
\label{fig:pre-cpx-2-3-2m}
\end{figure}
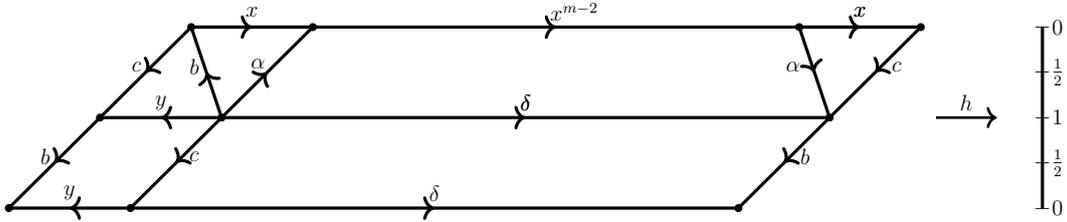

Considering a $[0,1/2]$-valued height function on this presentation complex $X_H$, we obtain a splitting of $\Art_{2,3,2m}$, using a similar idea as in \cite[Section 4]{Jan22}.

\begin{theorem}\label{23even}
When $m \geq 3$, the Artin group $\mathrm{Art}_{2,3,2m}=A*_CB$  is an amalgamated product of free groups for $A=\langle x,y, \delta \rangle \cong F_{3}, B\cong F_{4}, C\cong F_{7}.$ Here $x=c(ba)c^{-1},y=bc,\delta = bc(ba)^m.$
\end{theorem}
\begin{proof}
We first put a height function $h: X_H \rightarrow [0,1/2]$ on the presentation complex  $X_H$ with the three horizontal lines all mapped to $1=0$, and the middle points between those lines mapped to $\frac{1}{2}$, see Figure \ref{fig:pre-cpx-2-3-2m}. Note that $t$ and $1-t$ are then identified, for any $t\in [0,1/2]$.

\begin{figure}[ht!]
\centering
\includegraphics[width=90mm]{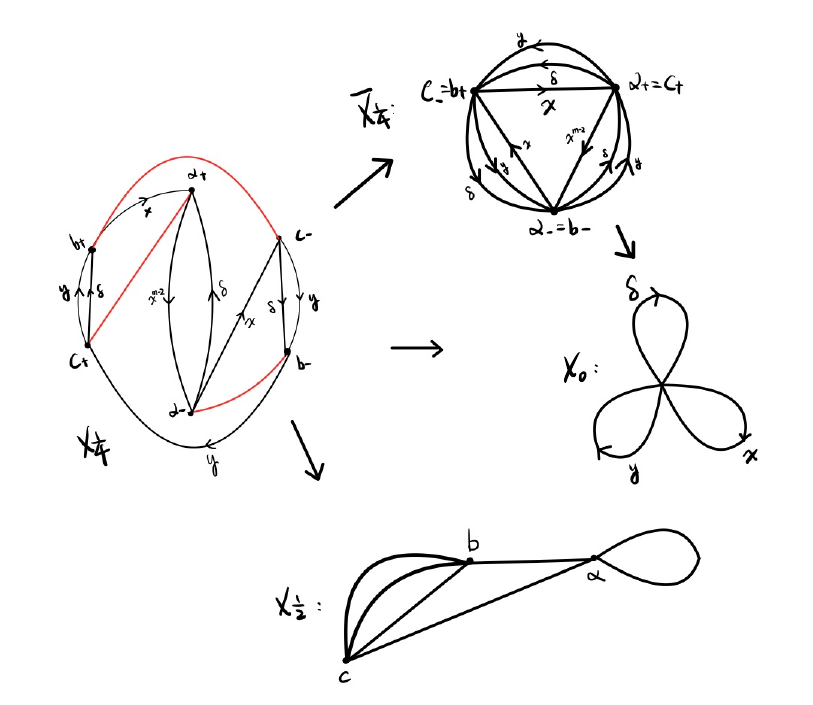}
\caption{Splitting of $A_{2,3,2m}$ }
\label{pic-23even}
\end{figure}

Fix $1/4<\epsilon < 1/2$. We now divide $X_H$ into three pieces: 
$$N_{0} = h^{-1}([0,\epsilon)), N_{\frac{1}{2}} = h^{-1}((1/2 - \epsilon,1/2]), \text{and } N_{\frac{1}{4}} = N_0\cap N_{\frac{1}{2}}.$$
They are tubular neighborhoods of  $X_{0} = h^{-1}(\{0\}), X_{\frac{1}{2}} = h^{-1}(\frac{1}{2}), \text{and } X_{\frac{1}{4}} = h^{-1}(\frac{1}{4})$, respectively. See Figure \ref{pic-23even} for a picture of $X_0$, $X_{\frac{1}{2}}$ and $X_{\frac{1}{4}}$. Note that $X_{0}, X_{\frac{1}{4}} $ and $X_{\frac{1}{2}}$ are graphs. Define $$A = \pi_1{X_0} = \pi_1(N_0), B = \pi_1(X_{\frac{1}{2}}) = \pi_1(N_{\frac{1}{2}})$$ and $C  = \pi_1(X_{\frac{1}{4}}) = \pi_1(N_{\frac{1}{4}})$. Note that $A$ is a free group of rank $3$, $B$ is a free group of rank $4$, and $C$ is a free group of rank $7$. To prove the theorem, it suffices now to show that the composition of the  maps 
$$X_{\frac{1}{4}} \hookrightarrow N_{0} \rightarrow X_{0},$$
and 
$$X_{\frac{1}{4}} \hookrightarrow N_{\frac{1}{2}} \rightarrow X_{\frac{1}{2}}, $$
induces injective maps on fundamental groups. Now in the graph of $X_{\frac{1}{4}}$, we have indicated how the edges are mapped to that of $X_0$.  In fact,  one  first collapses the red edges in $X_{\frac{1}{4}}$, in which case $X_{\frac{1}{4}}$ becomes the graph $\bar{X}_{\frac{1}{4}}$, then maps it to $X_0$ using the edge labeling. The key observation now is that the map from $\bar{X}_{\frac{1}{4}}$ to $X_0$ has no folding edge and thus is a combinatorial immersion. Therefore, the induced map of their fundamental groups is injective, i.e.~the map from $C$ to $A$ is injective, by Lemma \ref{stallings}. On the other hand, the map from $X_{\frac{1}{4}}$ to $X_{\frac{1}{2}}$ is a $2$-sheet covering. Thus the induced map from $C$ to $B$ is again injective.  Therefore, we have proved that $\mathrm{Art}_{2,3,2m}$ splits as $A\ast_{C}B$ where $A$ is a free group of rank $3$, $B$ is a free group of rank $4$ and $C$ is a free group of rank $7$.

\end{proof}

\section{Non-splitting of the Artin group $(2,3,2m+1)$}\label{sec:split-23odd}

In this section, we will study the isometric actions of Artin groups on simplicial trees and finish the proof of Theorem \ref{th1}.  For an isometry $g: T \rightarrow T$ of a simplicial tree $T$, the minimal set
$$\mathrm{Min}(g)=\{x\in T: d(x,g x)=\inf_{y\in T}\{d(y,g y)\}\},$$ 
is a subtree of $T$. When $d(x,gx)>0$ for some (hence any) $x\in \mathrm{Min}(g)$, the isometry $g$ is called a hyperbolic isometry of translation length $d(x,gx)$. In this case, the minimal set $\mathrm{Min}(g)$ is a geodesic line, called the geodesic axis of $g$. When $d(x,gx)=0$ for some (hence any) $x\in \mathrm{Min}(g)$, the element $g$ is called elliptic. In this case, we denote $\mathrm{Fix}(g)=\mathrm{Min}(g)=\{x\in T \mid gx=x \}$, the fixed point set of $g$.

Let $$A_{I_{2m+1}} = \langle a,b \mid (ab)^m a = b (ab)^m \rangle$$ be the Artin group of odd dihedral type. We will keep on using the following fact, see for example \cite[Proof of Lemma 3.4]{Wu22} for a proof.

\begin{Fact}\label{fact-dih-amalg}
    $A_{I_{2m+1}}$ has the amalgamated product structure
    $$\langle x\rangle \ast_{x^{2m+1}= s^2} \langle s\rangle, $$
    where $x= ab$ and $s = (ab)^ma$.  
\end{Fact}

Now let $$\mathrm{Art}_{2,2n+1,2m+1}=\langle
a,b,c \mid ac=ca,(bc)^{n}b=c(bc)^{n},(ab)^{m}a=b(ab)^{m}\rangle $$ be the triangle
Artin group. Let $x=ab,y=cb \in A_{2,2n+1,2m+1}.$ Note that $x^{m}a=bx^{m}$ and $y^{n}c=by^{n},$
which implies that $a,b,c$ are conjugate.

The actions of the dihedral Artin group $A_{I_{2m+1}}$ on trees are classified in the following lemma.

\begin{lemma}
\label{dihedr}Let $A_{I_{2m+1}}$ act on a tree $T$ by isometries. We then must have
one of the following.
\begin{enumerate}[label=(\arabic*).]
    \item \label{lem-dihedr-1} If $\mathrm{Fix}(ab)=\emptyset ,$ there is a geodesic line on which $%
A_{I_{2m+1}}$ acts by translations;
    \item If $\mathrm{Fix}(ab)\neq \emptyset $ and $\mathrm{Fix}(ab)\cap \mathrm{Fix%
}((ab)^{m}a)\neq \emptyset ,$ there is a global fixed point;
    \item If $\mathrm{Fix}(ab)\neq \emptyset $ and $\mathrm{Fix}(ab)\cap \mathrm{Fix%
}((ab)^{m}a)=\emptyset ,$ there is an invariant non-trivial subtree $T_{0}$, on which the action of $A_{I_{2m+1}}$ factors through an action of $\mathbb{Z}/(2m+1)\ast \mathbb{Z}/2$, where the image of $ab$ generates $\mathbb{Z}/(2m+1)$ and
the image of $(ab)^{m}a$ generates $\mathbb{Z}/2$. 

\end{enumerate}
\end{lemma}

   \begin{proof}
Note first that the element $(ab)^{2m+1}$ generates the center of $A=A_{I_{2m+1}}.$ If $ab$ has no
fixed points, the minimal set $\mathrm{Min}((ab)^{2m+1})$ is a geodesic
line, on which $A$ acts invariantly. This gives a group homomorphism $$f:A%
\rightarrow \mathrm{Isom}(\mathbb{R})=\mathbb{R}\rtimes \mathbb{Z}/2.$$ Since 
$(ab)^{m}a=b(ab)^{m},$ we know that $\mathrm{Im}(f)$ lies in the translation
subgroup $\mathbb{R} \leq \mathrm{Isom}(\mathbb{R})$ by the following argument. Note first that the elements $a,b$ are conjugate.  If both $a$ and $b$ act by
reflections on $\mathrm{Min}((ab)^{2m+1}),$ we can assume  that $%
f(a)(t)=-t+a_{0},f(b)(t)=-t+b_{0}$ for some real numbers $a_{0},b_{0}$, where $%
t\in $ $\mathbb{R}$. The relator $(ab)^{m}a=b(ab)^{m}$ implies that 
\begin{eqnarray*}
(ab)^{m}a(t) &=&-t+a_{0}+m(a_{0}-b_{0}) \\
&=&b(ab)^{m}(t)=-t-m(a_{0}-b_{0})+b_{0}
\end{eqnarray*}%
and $a_{0}=b_{0}.$ This would imply $ab$ acts trivially on $\mathrm{Min}%
((ab)^{2m+1}),$ which is a contradiction. 

If $ab$ has a fixed point, then $A$
acts invariantly on the fixed point set $\mathrm{Fix}((ab)^{2m+1}),$ which is a
subtree. The action factors through $A/\langle (ab)^{2m+1}\rangle =\mathbb{Z}%
/(2m+1)\ast \mathbb{Z}/2,$ where $x:=ab$ is the generator of $\mathbb{Z}%
/(2m+1)$ and $s:=(ab)^{m}a$ is the generator of $\mathbb{Z}/2.$  In this case, the element $s$  (as a torsion element in the quotient group) must also have a fixed point. 

If $%
\mathrm{Fix}(ab)\cap \mathrm{Fix}((ab)^{m}a)\neq \emptyset ,$ any point $z\in \mathrm{Fix}(ab)\cap
\mathrm{Fix}((ab)^{m}a)$ will be a global fixed point of $A$ since $ab$ and $(ab)^{m}a$ generate $A$.

If $%
\mathrm{Fix}(ab)\cap \mathrm{Fix}((ab)^{m}a)=\emptyset ,$ there are unique points $p_0\in
\mathrm{Fix}((ab)^{m}a),q_0\in \mathrm{Fix}(ab)$ such that 
\begin{equation*}
d(p_0,q_0)=\min \{d(p,q):p\in \mathrm{Fix}((ab)^{m}a),q\in \mathrm{Fix}(ab)\},
\end{equation*}
since the fixed point sets are closed convex subtrees. Note that the intersection of
the segments $[q_0,p_0]\cap \lbrack b(ab)^{m}q_0,p_0]=\{p_0\},$ by
the choice of $p_0.$  In fact, any $p\in \lbrack sq_{0},p_{0}]\cap
\lbrack q_{0},p_{0}]$ must be fixed by $s$ and $d(p,q_{0})\leq d(p_{0},q_{0})$, thus $p=p_0$ by the minimality of $d(p_0,q_0)$. 
Similarly, we have 
$[q_{0},xp_{0}]\cap [q_{0},p_{0}]=\{q_{0}\}.$  This further implies that the union $%
T_{0}:=\cup _{g\in A}[gp_0,gq_0]$ is connected,  hence an invariant non-trivial subtree, see Figure \ref{actiontree}. 

\begin{figure}
\centering

\begin{tikzpicture}[line width=1.0pt, scale = 1.5] 

\begin{scope}[decoration={
    markings,
    mark=at position .6 with {\arrow{}}}]
   \draw[-] (-1,0) --  (0,0);
   \draw[-] (0,0) --  (1,0);
   \draw[-] (-2,1) --  (-1,0);
   \draw[-] (-2,-1) --  (-1,0);
   \draw[-] (1,0) --  (2,1);
   \draw[-] (1,0) --  (2,-1);
  
\filldraw (0,0) circle (2.5pt);
 \filldraw       (1,0) circle (2.5pt);
 \filldraw       (-1,0) circle (2.5pt);
 \filldraw       (-2,-1) circle (2.5pt);
  \filldraw       (-2,1) circle (2.5pt);
   \filldraw       (2,-1) circle (2.5pt);
    \filldraw       (2,1) circle (2.5pt);

\draw (0, -0.3)  node[text=black, scale=1.0]{$p_0$};
\draw (-1 - 0.5, 0)  node[text=black, scale=1.0]{$q_0$};

\draw (-0.8, -0.4)  node[text=black, scale=1.0]{$x$};
\draw (-2, 1+0.3)  node[text=black, scale=1.0]{$a p_0 = x^{-m} p_0$};
\draw (-2, -1-0.3)  node[text=black, scale=1.0]{$b^{-1} p_0 = x^{m} p_0$};
 \draw (1 + 0.7, 0)  node[text=black, scale=1.0]{$b q_0 = s q_0$};
\draw (2, -1-0.3)  node[text=black, scale=1.0]{$a^{-1} p_0$};
\draw (2, 1 + 0.3)  node[text=black, scale=1.0]{$b p_0 = sx^{-m} p_0$};
\draw (0, 0.4)  node[text=black, scale=1.0]{$s$};
\draw (0, -0.3)  node[text=black, scale=1.0]{$p_0$};

 \draw [black,<->,thick,domain=0:180] plot ({0.3*cos(\x)}, {0.3*sin(\x)});

\draw [black,->,thick,domain=0:-120] plot ({-1 + 0.3*cos(\x)}, {0.3*sin(\x)});

\end{scope}

\end{tikzpicture}
\caption{The action of $A_{I_{2m+1}}$ on a tree $T$, where $x=ab$ acts as a rotation along $q_0$ and $s=(ab)^ma$ acts a reflection along $p_0$.}
\label{actiontree}
\end{figure}
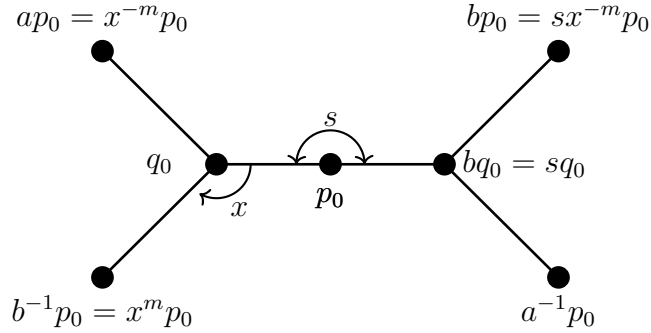
\end{proof}

Now we study the isometric actions of triangle Artin groups on trees.
\begin{lemma}
\label{lem2}Suppose that $\mathrm{Art}_{2,2n+1,2m+1}$ acts on a tree $T$ such that 
the fixed point set $\mathrm{Fix}(a)$ of $a$ is nonempty. Then the action has a global
fixed point.
\end{lemma}

\begin{proof}
Since $a,b,c$ are conjugate to each other, we know that $b,c$ have fixed
points as well.  Note that $x ^{2m+1}=(ab)^{2m+1}$ is in the center of the
subgroup $\langle a,b\rangle .$ If the minimal set $\mathrm{Min}(x^{2m+1}) \neq \mathrm{Fix} (x^{2m+1})$ is a geodesic line $l=\mathbb{R},$ the group $\langle
a,b\rangle $ acts on $l=\mathbb{R}$ by translations according to Lemma \ref%
{dihedr} \ref{lem-dihedr-1} However, the nonvanishing of the translation length of $x^{2m+1}$ implies that $a$ acts without fixed points, which is a
contradiction to the assumption that $a$ has a fixed point. Therefore, we
have $\mathrm{Min}(x^{2m+1})=\mathrm{Fix}(x^{2m+1}),$ a
subtree on which $\langle a,b\rangle /\langle x^{2m+1}\rangle
\cong \mathbb{Z}/(2m+1)\ast \mathbb{Z}/2$ (generated by the images of $x=ab$ and $%
s=(ab)^{m}a$) acts invariantly. Note that elements $a=x^{-m}((ab)^{m}a), b=b(ab)^{m} x^{-m}$ have nonempty fixed point sets in $T$, thus it must have bounded orbit for the action on the subtree $\mathrm{Fix}(x^{2m+1})$ which implies they must have fixed points in $\mathrm{Fix}(x^{2m+1})$ \cite[Proposition II.6.7]{BH99}. Since the  $ab =x$  in the group $\langle a,b\rangle /\langle x^{2m+1}\rangle$ is of finite order, the elements $a,b,ab$ all have non-empty fixed point sets in $\mathrm{Fix}(x^{2m+1})$.  The Helly theorem then
implies that the whole subgroup $\langle a,b\rangle $ has a fixed point (see for example 
\cite[p.64, Proposition 26]{Se03}). A similar argument shows that the
subgroups $\langle a,c\rangle ,\langle c,b\rangle $ have fixed points as
well. Therefore, the whole Artin group $\mathrm{Art}_{2,2n+1,2m+1}$ has a fixed point
by applying the Helly theorem again.
\end{proof}

\begin{lemma}\label{abhasnofixed}
Suppose that $\mathrm{Art}_{2,2n+1,2m+1}$ acts on a tree $T$ with the element $ab$ or $%
cb$ having no fixed points. There exists a geodesic line $l$ in $T$ on which 
$\mathrm{Art}_{2,2n+1,2m+1}$ acts by translations.
\end{lemma}

\begin{proof}
Assuming that $ab$ has no fixed points, the central element $(ab)^{2m+1}$ of $%
A_{I_{2m+1}}=\langle a,b:(ab)^{m}a=b(ab)^{m}\rangle$
has its minimal set a geodesic line $l=\mathbb{R}$, on which the group $%
A_{I_{2m+1}}$ acts by translations
according to Lemma \ref{dihedr} \ref{lem-dihedr-1} Since $ac=ca,$ the element $c$ acts
invariantly on $l$ as well. But since $b,c$ are conjugate,
we know that $c$ acts on $l$ by translations. Since the group $\mathrm{Art}_{2,2n+1,2m+1}$ is generated by $a,b,c$, the claim is proved.
\end{proof}

\begin{theorem}\label{thm-tree-action-clf}
\label{oddtree}Suppose that $\mathrm{Art}_{2,2n+1,2m+1}$ acts on a simplicial tree $T$
minimally (i.e.~there are no proper $\mathrm{Art}_{2,2n+1,2m+1}$-invariant subtrees). We have one of the
following.

\begin{enumerate}[label=\arabic*).] 
\item  If $\mathrm{Fix}(a)\neq \emptyset$, the tree $T$ is a point.

\item  If $\mathrm{Fix}(a)=\emptyset $ and $\mathrm{Fix}(ab)=\emptyset $ (or $%
\mathrm{Fix}(cb)=\emptyset ),$ the tree $T$ is isometric to $\mathbb{R},$ on
which $A_{2,2n+1,2m+1}$ acts by translations.

\item  If $\mathrm{Fix}(a)=\emptyset $, $\mathrm{Fix}(ab)\neq \emptyset $ and $%
\mathrm{Fix}(cb)\neq \emptyset ,$ the tree  $T$ has a geodesic path $e=[p_{0},q_{0}]$ with vertex stabilizers 
\begin{eqnarray*}
G_{q_0} &\supset &\langle cb,ab\rangle , \\
G_{p_0} &\supset &\langle (cb)^{n}c,(ab)^{m}a \rangle
\end{eqnarray*}%
and the path stabilizer 
\begin{equation*}
G_{e}\supset \langle
(ab)(cb)^{-1},(cb)^{-n}(ab)^{m},(cb)^{n}(ab)^{-m},(cb)^{2n+1},(ab)^{2m+1}%
\rangle.
\end{equation*}
\end{enumerate}
\end{theorem}

\begin{remark}
    When $n,m\geq 2$, $\mathrm{Art}_{2,2n+1,2m+1}$ does split as an amalgamated product of  two free groups \cite[Corollary 2.9]{Jan21}. In particular, it acts minimally on a tree whose quotient is an edge.
\end{remark}

\begin{proof}[Proof of Theorem \ref{thm-tree-action-clf}]
The case of 1) is proved in Lemma \ref{lem2}, while 2) is proved in Lemma \ref{abhasnofixed}. It is enough now to consider the case when $ab,cb$
both have fixed points, but $a$ acts hyperbolically. Let $x=ab,y=cb,s_{1}=b(ab)^{m}=(ab)^{m}a,s_{2}=b(cb)^{n}=(cb)^{n}c.$  Note that  $x,s_1$ acts invariantly on the fixed point set $\mathrm{Fix}(x^{2m+1})$ and the fixed point sets $\mathrm{Fix}(x^{2m+1}),\mathrm{Fix}(s_1)$ are subtrees.

There are unique points $p_0\in \mathrm{Fix}((ab)^{m}a),q_0\in
\mathrm{Fix}(ab) $ such that $$d(p_0,q_0)=\min \{d(p,q):p\in \mathrm{Fix}((s_1),q\in
\mathrm{Fix}(x)\}.$$ 
By the choice of $p_{0},q_{0},$ we have that $[s_1q_{0},p_{0}]\cap \lbrack
q_{0},p_{0}]=\{p_{0}\}$. Indeed, any $p\in \lbrack s_1q_{0},p_{0}]\cap
\lbrack q_{0},p_{0}]$ must be fixed by $s_1$ and $d(p,q_{0})\leq d(p_{0},q_{0})$, thus $p=p_0$ by the minimality of $d(p_0,q_0)$. 
Similarly, we have 
$[q_{0},xp_{0}]\cap [q_{0},p_{0}]=\{q_{0}\}.$ In fact,  we even have $[q_{0},x^{-m} p_{0}]\cap [q_{0},p_{0}]=\{q_{0}\}$.  In more detail, $x^{2m+1} p_0 =(ab)^{2m+1} p_0= ((ab)^ma)^2 p_0 = p_0$ and the integers $m,2m+1$ are coprime. Now   for any point $p$ that lies in the intersection $[q_{0},x^{-m} p_{0}]\cap [q_{0},p_{0}]$, we have $x^{-m} p =p$ and $x p = x^{um+v(2m+1)} p =p$ for some integers $u,v$ as $m, 2m+1$ are coprime. This implies that $p$ lies in the $\mathrm{Fix}(x)$. By the choice of $q_0,$ we have $p=q_0.$

Furthermore, we have $$[ap_0,p_0]= [x^{-m} s_1 p_0, p_0]=[x^{-m} p_0, p_0]= [x^{-m} p_0,q_0] \cup [q_0, p_0].$$ Similarly, we have
\begin{equation}
    \begin{split}
      [a^{-1}p_0,p_0] &= [s_1^{-1} x^{m} p_0, p_0] =[s_1^{-1} x^{m} p_0, s_1^{-1}q_0] \cup [s_1^{-1}q_0, p_0],  \\
      [ bp_{0},p_{0}] &=[s_{1}x^{-m}p_{0},p_{0}] =[s_{1}x^{-m}p_{0},s_{1}q_{0}]\cup [
s_{1}q_{0},p_{0}], \\
[ b^{-1}p_{0},p_{0}] &=[x^{m}s_1^{-1}p_{0},p_{0}] =[x^{m}p_{0},q_{0}]\cup
[ q_{0},p_{0}].
    \end{split}
\end{equation}

This implies that the geodesic segments $[a p_0, a^{-1} p_0 ]=[a p_0,p_0] \cup [p_0, a^{-1} p_0], [ bp_{0},b^{-1}p_{0}] = [bp_{0},p_{0}]\cup [p_{0},b^{-1}p_{0}].$

Therefore, $$\cup _{i\in \mathbb{Z}}[a^{i}p_0,a^{i+1}p_0]= \mathrm{Min}(a), \cup _{i\in \mathbb{Z}}[b^{i}p_0,b^{i+1}p_0]= \mathrm{Min}(b),%
$$ are geodesic lines, see Figure \ref{actiontree}. Indeed, it follows from the observation that $[p_0,bp_0]\cap [b^{-1}p_0,p_0] =p_0$ and  $[p_0,ap_0]\cap [a^{-1}p_0,p_0] =p_0$. 
Similarly, for the dihedral Artin group $\langle b,c\rangle ,$ we have
geodesic lines $$\cup _{i\in \mathbb{Z}}[b^{i}p_0^{\prime},b^{i+1}p_0^{\prime}]= \mathrm{Min}(b),%
\cup _{i\in \mathbb{Z}}[c^{i}p_0^{\prime},c^{i+1}p_0^{\prime}]= \mathrm{Min}(c)$$ for unique points $p_0^{\prime}\in \mathrm{Fix}((cb)^{n}c),q_0^{\prime}\in \mathrm{Fix}(cb)$ such
that $$d(p_0^{\prime },q_0^{\prime })=\min \{d(x,y):p'\in
\mathrm{Fix}((cb)^{n}c),q'\in \mathrm{Fix}(cb)\}.$$ Since $ac=ca,$ we have $\mathrm{Min}(a)=\mathrm{Min}(c).$ By
the uniqueness of $\mathrm{Min}(b),$ we see from Figure \ref{actiontree} that $$[q_0,bq_0]=\mathrm{Min}(a)\cap
\mathrm{Min}(b)=\mathrm{Min}(c)\cap\mathrm{Min}(b)=[q_0^{\prime },bq_0^{\prime }].$$ If $q_0= b q_0^{\prime },$ we must have $b q_0 = q_0^{\prime }$ and $b^2 q_0 = q_0$, which is a contradiction to the assumption that $a,b$ are hyperbolic. Therefore, we have $q_0 = q_0^{\prime }$. Note that $%
p_0$ is the middle point of $[q_0,bq_0]$ and $p_0^{\prime }$ is the
middle point of $[q_0^{\prime },bq_0^{\prime }].$ Therefore, we have $%
p_0=p_0^{\prime }.$ Now since $a,b,c$ are conjugate, they must have the same translation length. But $a,b$ (resp. $b,c$) move $p_0$ to opposite directions, we deduce that the actions of $a,c$ on $\mathrm{Min}(a)$ are the same. 

 We have $x,y\in G_{q_0}.$
Since $xy^{-1}=ac^{-1}$ and $a,c$ have the same action on the minimal set, we have that $%
xy^{-1}p_0=p_0$ since $p_0 \in \mathrm{Min}(a)$. Note that $b=s_{1}x^{-m}=s_{2}y^{-n}$ and $%
b^{-1}=x^{m}s_{1}^{-1}=y^{n}s_{2}^{-1}.$ We have $%
b^{-1}p_0=x^{m}p_0=y^{n}p_0$ and $y^{-n}x^{m}p_0=p_0.$
Furthermore, we have $a=x^{-m}s_{1},c=y^{-n}s_{2},$ implying that $%
x^{-m}p_0=y^{-n}p_0,y^{n}x^{-m}p_0=p_0$ from $ap_0=cp_0.$ From $%
x^{2m+1}=s_{1}^{2},y^{2n+1}=s_{2}^{2},$ we also know that $x^{2m+1},y^{2n+1}\in
G_{e}$ since they fix the endpoints $p_0$ and $q_0$ of $e$.
\end{proof}

\begin{corollary}
\label{5.7}
When $m\neq 3k+1$ for any integer $k$, any isometric action of the Artin group $\mathrm{Art}_{2,3,2m+1}$ on
a simplicial tree either has a global fixed point or has an invariant
geodesic line. In particular, the Artin group $\mathrm{Art}_{2,3,2m+1}$ does not split
as a nontrivial graph of free groups when $m\neq 3k+1$.
\end{corollary}

\begin{proof}
Suppose first that there is a tree $T$ on which $\mathrm{Art}_{2,3,2m+1}$ acts minimally without fixed
points nor an invariant geodesic line. Apply Theorem \ref{oddtree} to our situation (recall that we are in the case $(2,2n+1,2m+1)$ for $2n+1 =3$),
there is a geodesic path $[p_0,q_0]$ with vertex
stabilizers 
\begin{equation*}
G_{q_0}\supset \langle cb,ab\rangle ,G_{p_0}\supset \langle
(cb)c,(ab)^{m}a,ac^{-1},(cb)^{-1}(ab)^{m},(cb)(ab)^{-m}\rangle .
\end{equation*}%
This implies that $(cb)^{-1}p_0=(ab)^{-m}p_0$ and $%
(cb)^{-1}p_0=(ab)^{-1}p_0,$ implying that $(ab)^{m-1}p_0=p_0.$ Since 
$\gcd (m-1,2m+1)=\gcd (m-1,2(m-1)+3)=1$ when $m\neq 3k+1,$ there are integers $%
k_{1},k_{2}$ such that $1=k_{1}(m-1)+k_{2}(2m+1).$ Therefore, we have $%
(ab)p_0=(ab)^{(m-1)k_{1}}(ab)^{(2m+1)k_{2}}p_0=p_0.$ This implies that 
$ab\in G_{p_0}$ and $a=(ab)^{-m}(ab)^{m}a\in G_{p_0}.$ By Lemma \ref%
{lem2}, the action of $A$ on $T$ has a global fixed, which is a contradiction.

Suppose now that $\mathrm{Art}_{2,3,2m+1}$ splits as a nontrivial graph of free groups. Then the action of $\mathrm{Art}_{2,3,2m+1}$ on the corresponding Bass--Serre tree can not have a global fixed point since $\mathrm{Art}_{2,3,2m+1}$ is not a free group. Without loss of generality, we can further assume that the tree is minimal. Then such a tree must be a line. By Theorem \ref{thm-tree-action-clf}, $\mathrm{Art}_{2,3,2m+1}$ acts on the line by translations. Since the translations of the line form an abelian group, the action of $\mathrm{Art}_{2,3,2m+1}$ on the line factors through its abelianization. But then the stabilizer of a vertex in the line would be the commutator subgroup which is not free. We postpone the proof of this fact to the last section, see Theorem \ref{th2} or Proposition \ref{nonfree}. This again leads to a contradiction. 
\end{proof}

\begin{corollary}
\label{5.8}
When $m=3k+1$ for some integer $k$, any isometric action of the Artin group $\mathrm{Art}_{2,3,2m+1}$ on a
simplicial tree either has a global fixed point or has an invariant geodesic line or has a vertex stabilizer $G_{p_0} > \langle s_1= (ab)^{m}a, x^3=(ab)^3 \rangle$. In particular, $\mathrm{Art}_{2,3,6k+3}$ does not split as any
nontrivial graph of free groups.
\end{corollary}

\begin{proof}
Suppose that the group $A_{2,3,2m+1}$ acts on a tree $T$ and the action is
minimal without fixed points or invariant geodesic lines. Theorem \ref%
{oddtree} implies that there is a geodesic path $e=[q_0,p_0]$ such
that 
\begin{equation*}
G_{e}\supset \langle xy^{-1},y^{-1}x^{3k+1},yx^{-3k-1},y^{3},x^{6k+3}\rangle 
\end{equation*}%
for $x=ab,y=cb.$ Note that $x^{-3k}=xy^{-1}(yx^{-3k-1})\in G_{e}.$
Therefore, the element $x^3= x^{6k+3 - 6k} \in G_e$.  However, in the stabilizer $G_{p_0}$, we have elements $s_1= b(ab)^m,x^3 $ lying in the dihedral group $\langle a,b \rangle$. Now  by Fact \ref{fact-dih-amalg}, $\langle a,b \rangle$ has the amalgamated product structure $ \langle s_1,x \mid x^{6k+3} = s_1^2\rangle$. This implies that the subgroup generated by $x^3$ and $s_1$ is not a free group, in fact it satisfies a nontrivial relation $s_1^2 = x^{2m+1}= x^{6k+3} = (x^3)^{2k+1}.$ This proves that $G_{p_0}$ is not a free group when $k>0$. When $k=0$, the Artin group $\mathrm{Art}_{2,3,6k+3}=\mathrm{Art}_{2,3,3}$ is of finite type, which is not isomorphic to a graph of free groups (see \cite[Corollary 2.11]{Jan21}).  Therefore, the group $\mathrm{Art}_{2,3,6k+3}$ does not split as any
nontrivial graph of free groups.
\end{proof}

\begin{remark}
Each case in Corollary \ref{5.8} can happen. The triangle Artin group $\mathrm{Art}_{2,3,6k+3}=\langle
a,b,c \mid ac=ca,bcb=cbc,(ab)^{3k+1}a=b(ab)^{3k+1}\rangle$ has an epimorphism onto $%
\mathbb{Z}/3\ast \mathbb{Z}/2$ explicitly given as the following. Using the
notation $x=ab,y=cb,s_{1}=(ab)^{3k+1}a,s_{2}=cbc,$ the group  
\begin{equation*}
\mathrm{Art}_{2,3,6k+3}=\langle x,y,s_{1},s_{2}\mid
x^{6k+3}=s_{1}^{2},y^{3}=s_{2}^{2},
s_{1}y^{-1}=x^{3k+1}y^{-1}s_{2}x^{-3k-1}s_{1}s_{2}^{-1}, s_2^{-1} s_1 = y^{-1} x^{3k+1}\rangle .
\end{equation*}%
Suppose that $\mathbb{Z}/3=\langle u\rangle ,\mathbb{Z}/2=\langle v\rangle $
for generators $u,v.$ The map $f:\mathrm{Art}_{2,3,6k+3}\rightarrow \mathbb{Z%
}/3\ast \mathbb{Z}/2$ defined by%
\begin{eqnarray*}
x,y &\longmapsto &u, \\
s_{1},s_{2} &\longmapsto &v,
\end{eqnarray*}%
is a surjective homomorphism.
\end{remark}

\bigskip
\begin{proof}[Proof of Theorem \protect\ref{th1}]
When $M$ is odd, the Artin group $\mathrm{Art}_{23M}$ does not split as a graph of free groups by Corollary \ref{5.7} and Corollary \ref{5.8}.
When $M>4$ is even, the group $\mathrm{Art}_{23M}$ is isomorphic to the amalgamated product $F_3 *_{F_7} F_4$ by Theorem \ref{23even}.

\end{proof}

\section{Poly-free, algebraically clean and thin}

In this section, we prove some useful lemmas related to poly-freeness preparing our proof of Theorem \ref{th3}.

We start with a basic fact on poly-freeness, which is in a similar flavor as Lemma \ref{2.6}. A residual-finite version was obtained by Wise \cite{Wi02}.

\begin{lemma}
\label{key}Let $G=F\ast _{H_{i}^{t_{i}}=H_{i}^{\prime }}$ be a multiple HNN
extension of a group $F$ where for each $i$, the stable letter $t_{i}$
conjugates the subgroup $H_{i}$ to the subgroup $H_{i}^{\prime }.$ Then $G$
is poly-free (resp. normally poly-free) if

1) $F$ is poly-free (resp. free);

2) Each isomorphism $H_{i}\rightarrow H_{i}^{\prime }$ given by $t_{i}$ can
be extended to an automorphism $\phi _{i}:F\rightarrow F.$

\end{lemma}

\begin{proof}
Let $F_{I}$ be the free group generated by free generators $x_i$ for all the index $i \in I$ and $F\rtimes
F_{I}$ be the semi-direct product given by the $\phi _{i}$'s. We have an obvious
homomorphism%
\begin{equation*}
f:G=F\ast _{H_{i}^{t_{i}}=H_{i}^{\prime }}\rightarrow F\rtimes F_{I},
\end{equation*}%
mapping $t_i$ to $x_i$ and $F$ to $F$. The $\ker f$ as a subgroup of $G$ acts on the Bass--Serre tree of $G$ associated with the multiple HNN extension.  The action has trivial
vertex stabilizers. Therefore, $\ker f$ is a free group. By Lemma \ref{2.6}, we know that $G$ is poly-free.

The above proof shows that there is an exact sequence%
\begin{equation*}
1\rightarrow K\rightarrow G\rightarrow F\rtimes F_{I}\rightarrow 1,
\end{equation*}%
where $K$ is free. Therefore, $G$ is normally poly-free when $F$ is free.
\end{proof}

\begin{definition}
    
A graph of groups is \emph{algebraically clean} provided that each vertex group is free
and each edge group is a free factor of its vertex groups. A splitting of a
group $G$ as a graph of groups $\Gamma _{G}$ is \emph{virtually algebraically clean}
provided that $G$ has a finite index subgroup $H$ whose induced splitting $%
\Gamma _{H}$ is algebraically clean.
\end{definition}

\begin{corollary}\label{cor-ac-pf}
\label{cor1}(algebraically clean $\Longrightarrow $ normally poly-free) Let $G$ split as a graph of
groups $\Gamma .$ Suppose that each vertex group is free and each edge group
is a free factor of its vertex groups. Then $G$ is normally poly-free.
\end{corollary}

\begin{proof}
Choose a maximal tree $T$ of $\Gamma.$ The group $\ast _{v\in V(T)}G_{v}$
is free. Since each edge group is a free factor of its vertex groups, the graph of groups $%
\langle G_{v},v\in T\rangle $ over $T$ is free (if $T$ is infinite, start
the proof from a vertex $v_0$ and contracting edges nearby. Explicitly, for any vertex $v \neq v_0$, let $e_v$ be the unique edge pointing to $v_0$ and starting from $v$. Since $G_{e_{v}}$ is a free factor of $G_v$, we choose $S_v$ as a free generating set of the complement of $G_{e_{v}}$ in $G_v$. Let $S$ be a free-generating set of $G_{v_0}$. It can be directly checked that $S \cup_{v\in T \backslash v_0} S_{v} $ is a free generating set of $%
\langle G_{v},v\in T\rangle $). Note that $G$ is a
multiple HNN extension of $\langle G_{v},v\in T\rangle .$ For each edge $%
e\in \Gamma \backslash T$ connecting vertices $v,w\in T,$ we have $%
t_{e}G^*_{v}t_{e}^{-1}=G^*_{w},$ where $G^*_{v},G^*_{w}$ are free factors of $G_v,G_w$, respectively. Suppose that $$\langle G_{v},v\in T\rangle
=G^*_{v}\ast F_{V^{\prime }}=G^*_{w}\ast F_{W^{\prime }}$$ for free groups $%
G^*_{v}=F_{V},G^*_{w}=F_{W},$ $F_{V^{\prime }},F_{W^{\prime }},$ where $V,W,$ $%
V^{\prime },W^{\prime }$ are free generating sets. If $V^{\prime },W^{\prime
}$ have the same cardinality (for example when $G_{v}$ is finitely
generated), we have an obvious extension of $t$ to the whole free group $%
\langle G_{v},v\in T\rangle $. If $V^{\prime },W^{\prime }$ do not have the
same cardinality, consider the free product $$\langle G_{v},v\in T\rangle
\ast \langle G_{v},v\in T\rangle =G^*_{v}\ast F_{V^{\prime }}\ast G^*_{w}\ast
F_{W^{\prime }}=G^*_{w}\ast F_{W^{\prime }}\ast G^*_{v}\ast F_{V^{\prime }}.$$
Note that $V^{\prime }\cup W\cup W^{\prime }$ and $W^{\prime }\cup V\cup
V^{\prime }$ have the same cardinality. Therefore, the $t_{e}$ can be
extended to be an automorphism of $\langle G_{v},v\in T\rangle \ast \langle
G_{v},v\in T\rangle .$ We have obtained a multiple HNN extension $G^{\prime
}\ast _{t_{e}G^*_{v}t_{e}^{-1}=G^*_{w},e\in \Gamma \backslash T},$ into which
the group $G$ embeds. Lemma \ref{key} shows that $G^{\prime }\ast
_{t_{e}G^*_{v}t_{e}^{-1}=G^*_{w},e\in \Gamma \backslash T}$ hence $G$ are normally poly-free.
\end{proof}

\begin{remark}
We assume neither the graph $\Gamma $ is finite nor the edge groups are
finitely generated. The residual finiteness version of the corollary
is not true under this weaker assumption, as Baumslag constructs a
non-residually finite free-by-cyclic group $F_{\infty }\rtimes \mathbb{Z}.$
See Wise's remark \cite[Example 3.5]{Wi13}.
\end{remark}

An action of a group $G$ on a tree $T$ is thin provided that the stabilizer
of any embedded length 2 path in $T$ is finite. A graph of groups $\Gamma $
is thin if the action of $\pi _{1}(\Gamma) $ on the corresponding Bass-Serre tree is thin. Wise
\cite{Wi02} proves every thin finite graph of finitely generated virtually free
groups is virtually algebraically clean.

A subgroup $M<G$ is malnormal provided that for each $g\in G\backslash M$,
the intersection $M\cap gMg^{-1}$ is trivial. Similarly, $M$ is almost
malnormal provided that for each $g\in G\backslash M$ the intersection $%
M\cap gMg^{-1}$ is finite.

\begin{corollary}
\label{cor2}Let $G=A\ast _{M}B$ be an amalgamated free product, where $A$
and $B$ are virtually free, and $M$ is a finitely generated almost malnormal
subgroup of $A$ and $B$. Then $G$ is virtually normally poly-free.
\end{corollary}

\begin{proof}
By Wise \cite[Lemma 2.3]{Wi02} , a graph of groups $\Gamma $ is thin if and only
if the following two conditions hold: 1) for each end of an edge incident at 
$v$, the corresponding subgroup of $G_{v}$ is almost malnormal; 2) for each
pair of ends of edges of $\Gamma $ incident at $v$, any conjugates of the
subgroups of $G_{v}$ corresponding to these ends of edges have finite
intersection. This in particular implies our group $G$ is thin. Therefore, $G$ is virtually algebraically clean by \cite[Theorem 11.3]{Wi02}. By Corollary \ref{cor1}, the group $G$ is virtually normally poly-free.
\end{proof}

Note that in general an HNN extension of a free group may not be virtually poly-free (see Lemma \ref{lem1} below). An amalgamated free product of free groups could even be a simple group \cite{BM01}. Thus  in order to have a virtually poly-free
amalgamated free product $A\ast _{M}B,$ one has to put some additional
assumptions like those in Lemma \ref{key} and Corollary \ref{cor2}.

\begin{lemma}
\label{lem1}The  Baumslag-Solitar group $BS(m,n)=\langle
a,t:ta^{n}t^{-1}=a^{m}\rangle $ is virtually poly-free if and only if $m=\pm n.$
\end{lemma}

\begin{proof}
When $m=\pm n,$  Lemma \ref{key} implies that $BS(m,n)$ is normally poly-free. Conversely, let $K$ be the kernel of the group homomorphism $\phi
:BS(m,n)\rightarrow \mathbb{Z}$ given by $a\rightarrow 0$, $t\rightarrow 1.$
Note that $K$ is a graph of groups where the graph is a line. In fact,  $K$ has
the presentation  $\langle a_{i}:a_{i}^{m}=a_{i+1}^{n},i\in \mathbb{Z}\rangle
.$ Note that any group homomorphism $f:K\rightarrow \mathbb{Z}$ is trivial
by the following argument. We have $f(a_{i}^{m})=f(a_{i+1}^{n}),f(a_{i+1})=%
\frac{m}{n}f(a_{i})=\cdots =(\frac{m}{n})^{i+1}f(a_{0}).$ But for a fixed $%
f(a_{0})\in \mathbb{Z}\backslash \{0\},$ we cannot find $f(a_{i+1})$ with
this property for sufficiently large $i$ if $m\neq \pm n$. This implies $K$ is not indicable, hence it can not be poly-free. Therefore, $BS(m,n)$ is not poly-free since the subgroup $K$ is not. For a general finite-index subgroup $G<BS(m,n)$, there exist positive integers $k,l$ such that the elements $t^k,a^l \in G$. We have  $$t^k a^{n^k l} t^{-k} = a^{m^k l}.$$ 
A similar argument shows that $G$ is not poly-free if $m\neq \pm n$. Hence $BS(m,n)$ is not virtually poly-free.
\end{proof}

Let now $A\ast _{C}B$ be an amalgamated product of free groups $A,B$ over the free
group $C,$ with injections $\rho _{1}:C\hookrightarrow A,\rho
_{2}:C\hookrightarrow B.$ Let  $A_{\rho _{1}}\subseteq A,A_{\rho
_{2}}\subseteq B$ be oppressive sets with respect to some combinatorial
immersions of graphs inducing injections $\rho _{1},\rho _{2}.$ Similarly, let $A\ast _{C,\beta }$ be the HNN extension of a free group $A$ induced by the injection $\beta: C \rightarrow A$.  Denote by $A_{C\hookrightarrow A},A_{\beta (C)\hookrightarrow A}$ the oppressive sets from some graph immersions that induce $C\hookrightarrow A, \beta (C)\hookrightarrow A.$

The following is a generalization of \cite[Prop. 2.7]{Jan22}, where we
don't require that $K$ to be finite.

\begin{lemma}
\label{seplem}Let $K$ be a poly-free group. Suppose that

1) $\phi :A\ast _{C}B\rightarrow K$ is a group homomorphism such that $\phi
|_{A}$ separates  $C$ from $A_{\rho _{1}}$, and $\phi |_{B}$ separates $C$ from $%
A_{\rho _{2}}$, or

2) $\phi :A\ast _{C,\beta }\rightarrow K$ is a group homomorphism such that $%
\phi |_{A}$ separates $A_{C\hookrightarrow A},A_{\beta (C)\hookrightarrow A}$
from $C.$

Then $\ker \phi $ is an algebraically clean graph of free groups. In particular, the
amalgamated product $A\ast _{C}B$ and the HNN extension $A\ast _{C,\beta }$ are poly-free.
\end{lemma}

\begin{proof}
The $\ker \phi $ acts on the Bass--Serre tree $T$ of $A\ast _{C}B,$ with edge
stabilizers are conjugates of $\ker \phi |_{C}$ and the vertex stabilizers are conjugates of $\ker \phi |_{A},\ker \phi |_{B}.$ By the assumption,
Lemma \ref{2.4} implies that $\ker \phi |_{C}$ is a free factor of $\ker
\phi |_{A},\ker \phi |_{B}$ respectively. This further implies that, for the action of $\ker\phi$ on $T$, every edge group stabilizer is a free factor in each respected vertex stabilizer. Therefore, $\ker \phi $ is an algebraically clean
graph of free groups. Lemma \ref{cor1} implies that $\ker \phi$ is poly-free. Therefore, $A\ast _{C}B$, $A\ast _{C,\beta }$ are poly-free.
\end{proof}

\section{Poly-freeness of triangle Artin groups}

We prove Theorem \ref{th3} in this section. Recall that Jankiewicz proves in  \cite{Jan22, Jan21}  that most triangle Artin groups split as graphs of free groups. Abusing the notation, we may sometimes use the same symbol to denote both a graph map and its induced map on the fundamental groups.  We summarize her results as follows: 

\begin{theorem}\label{thm-split-sum}
\label{7.3}
Let $\mathrm{Art}_{MNP}$ be the triangle Artin group. 

\begin{enumerate}
\item[(i)] \cite[Corollary 4.13]{Jan22} When $M,N,P\geq 3,$ the group $\mathrm{Art}_{MNP}=A\ast _{C}B$
where $A\cong F_{3},B\cong F_{4},C\cong F_{7}$ and $[B,C]=2.$ The map $%
C\rightarrow A$ is induced by the  map $\rho_1
:X_{C}\rightarrow X_{A}$ in Figure \ref{pic-333-great}, and the homomorphism  $C\rightarrow B$
is induced by the quotient $\rho_2$ of the graph  $X_{C}$ by a 
rotation of angle $\pi $;

\begin{figure}[ht]
\centering
\includegraphics[width=.5\textwidth]{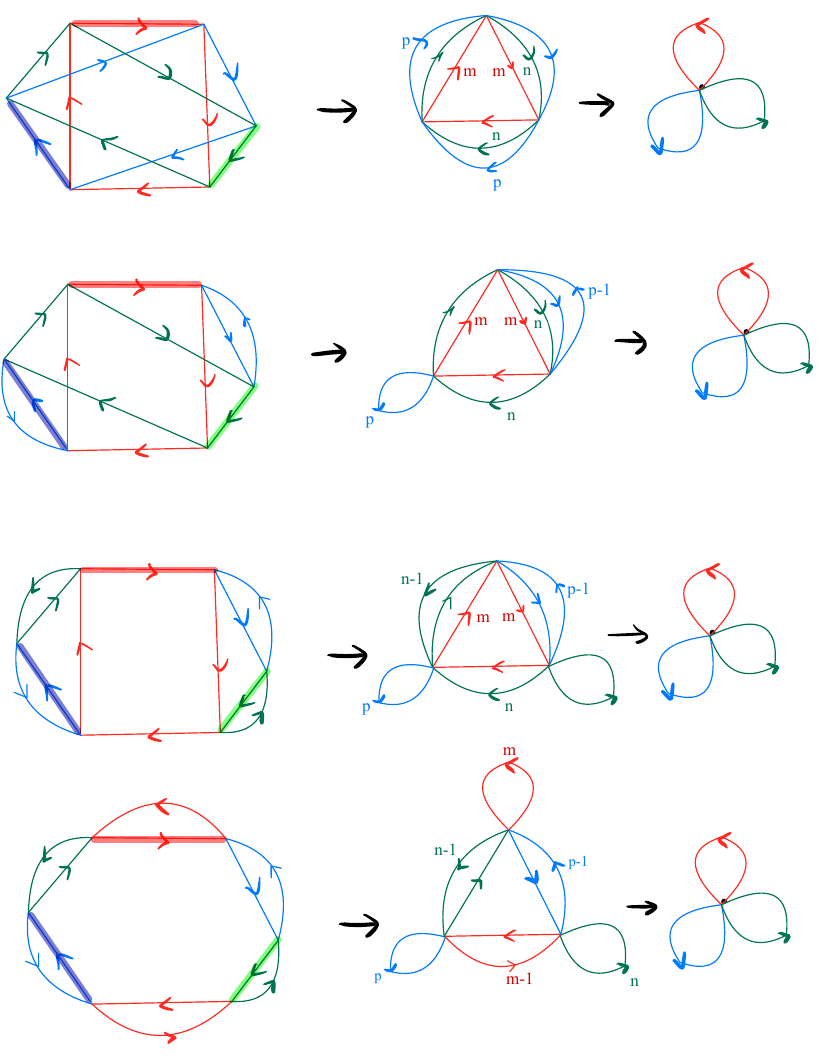}
\caption{The map $\rho_1: X_C \rightarrow \Bar{X}_C \rightarrow X_A $ when (1) none, (2) one, (3) two or (4) all of $M, N, P$ are even, respectively. In the picture, $M = 2m$ or $2m + 1, N = 2n$ or $2n + 1$, and $P = 2p$ or $2p + 1$. The edge labeled by a number $k$ is a concatenation of k edges of the given color.
The thickened edges in $X_C$ are the ones that collapsed to a vertex in $\bar{X}_C$.}
\label{pic-333-great}
\end{figure}

\item[(ii)] \cite[Proposition 2.4]{Jan21} When $P=2,$at least one of $M,N\geq 4$ is odd, the group $%
\mathrm{Art}_{MNP}=A\ast _{C}B$ where $A\cong F_{2},B\cong F_{3},C\cong F_{5}
$ and $[B,C]=2.$ The homomorphism $C\rightarrow A$ is induced by the  map $\rho_1 :X_{C}\rightarrow X_{A}$ in Figure \ref{pic-2-odd-odd}, and the homomorphism $%
C\rightarrow B$ is induced by the quotient $\rho_2$ of the graph $X_{C}$ by a 
rotation of angle $\pi $ (ignoring the label and orientation here);

\begin{figure} [ht]
\centering
\includegraphics[width=.5\textwidth]{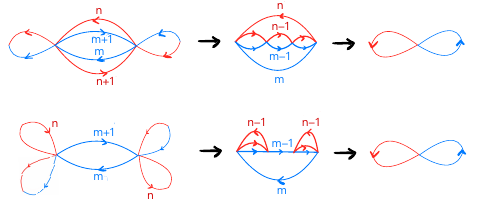}
\caption{The map $\rho_1:X_C \rightarrow \bar{X_C} \rightarrow X_A$ when $P = 2, M = 2m+1 \geq
5$, and (top) $N = 2n + 1 \geq 5$, (bottom) $N = 2n \geq4$, respectively. }
\label{pic-2-odd-odd}
\end{figure}

\item[(iii)] \cite[Corollary 2.9]{Jan21} When $P=2,$ both $M,N\geq 4$ are even, the group $\mathrm{Art}%
_{MNP}=A\ast _{B}$ where $A\cong F_{2},B\cong F_{3}.$ The two maps $%
B\rightarrow A$ are induced by the maps $\rho _{1},\rho
_{2}:X_{B}\rightarrow X_{A}$ in Figure \ref{2-e-e}.
\end{enumerate}
\end{theorem}

\begin{figure}[ht]
\centering
\includegraphics[width=.7\textwidth]{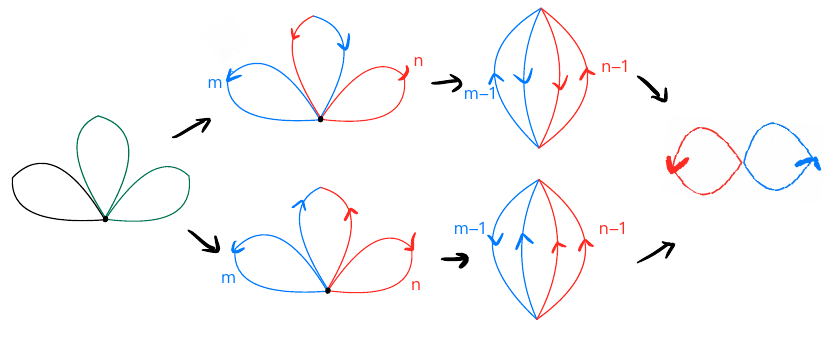}
\caption{The map $\rho_i: X_B \xrightarrow{=} X_B \rightarrow \bar{X}_B \rightarrow X_A,i=1,2,$ when $P = 2, M \geq
4, N\geq 4$ are even.}
\label{2-e-e}
\end{figure}

Based on these splittings, Jankiewicz proves the following.

\begin{lemma}
\label{lemkasia}
Suppose that
\begin{enumerate}
\item $M,N,P\geq 4$ and $(M,N,P)\neq (2m+1,4,4)$ (up to permutation), or

\item $M=2,N,P\geq 4$ and at least one of  $N,P$ is even. 
\end{enumerate}
Then the Artin group $\mathrm{Art}_{MNP} $ virtually splits as an algebraically clean
graph of finite-rank free groups.
\begin{proof}
(1) is proved by \cite[Cor. 5.7 and Cor. 5.12]{Jan22}. (2) is
contained in the proof of \cite[Theorem B]{Jan21}. In fact, \cite[Theorem B]{Jan21} is proved using \cite[Theorem 3.2]{Jan21} and \cite[Theorem 3.3]{Jan21}. But both theorems are proved by first showing that the groups satisfying conditions in the theorems virtually split as an algebraically clean graph of finite-rank free groups. 
\end{proof}
\end{lemma}

Combined Lemma \ref{lemkasia} with Corollary \ref{cor1}, we have the following:

\begin{corollary}\label{cor:vnpf}
Suppose that
\begin{enumerate}
\item $M,N,P\geq 4$ and $(M,N,P)\neq (2m+1,4,4)$ (up to permutation), or
\item $M=2,N,P\geq 4$ and at least one of  $M,P$ is even. 
\end{enumerate}
Then the Artin group $\mathrm{Art}_{MNP}$ is virtually normally poly-free.
\end{corollary}

Recall that given any Artin group $ \mathrm{Art}_{MNP}$, one can define its corresponding Coxeter group by
\begin{equation*}
\Delta_{MNP} =\langle
a,b,c:(ab)^{M},(bc)^{N},(ca)^{P},a^{2}=b^{2}=c^{2}=1\rangle.
\end{equation*}%
Let $\phi :\mathrm{Art}_{MNP}\rightarrow \Delta_{MNP} $  be the natural
surjection. Note that here we are abusing the notation to use $a,b,c$ denoting both the standard generators of $\mathrm{Art}_{MNP}$ and their images in $\Delta_{MNP}$.  We call its kernel the pure Artin group and denote it by $\mathrm{PA}_{MNP}$. 

\begin{theorem}\label{mthm-poly-free}
Let $ \mathrm{Art}_{MNP}$ be the triangle Artin group labeled by $M\geq
3,N\geq 3,P\geq 3$ or $M=2,N\geq 4,P\geq 4$ or $M=2,N=3,P\geq 6$ is even. The pure Artin group $\mathrm{PA}_{MNP}$ is an
algebraically clean graph of free groups (i.e.~each vertex subgroup is a free group and
each edge subgroup is a free factor of its vertex subgroups). Thus,
the Artin group $ \mathrm{Art}_{MNP}$ is virtually poly-free.
\end{theorem}

\begin{proof}
Note first that $\Delta_{MNP} $ is either a hyperbolic triangle group acting
properly and cocompactly on the upper half-plane $\mathbb{H}^{2}$ or an Euclidean triangle
group acting properly and cocompactly on $\mathbb{R}^{2}$. Since any finitely generated
Coxeter group is virtually a subgroup of a right-angled Artin group \cite[Corollary 1.3]
{HW10} and every right-angled Artin group is normally poly-free
(see for example \cite[Theorem A]{Wu22}), the group $\Delta_{MNP} $ is virtually normally poly-free. We
prove the theorem  by showing that $\mathrm{PA}_{MNP}$ is an algebraically clean graph of free groups, which implies that $ \mathrm{PA}_{MNP} = \ker \phi$ is poly-free using Lemma \ref{seplem},  thus $\mathrm{Art}_{MNP}$ is virtually poly-free.

\begin{enumerate}[label= Case \arabic*.]
\item  $M\geq 3,N\geq 3,P\geq 3.$ By Theorem \ref{7.3},
the Artin group $\mathrm{Art}_{MNP}$ splits as an amalgamated product $F_{3}\ast
_{F_{7}}F_{4}.$ The splitting is induced by a decomposition of a
presentation complex $X_{0}\cup X_{1/2}$ with $X_{0}\cap X_{1/2}=X_{1/4}.$
The injection $\rho _{2}:F_{7}\rightarrow F_{4}$ is induced by the
double-covering $f$ (rotating the plane by $\pi $) in Figure \ref{case1}. By Lemma \ref{seplem}, it suffices to prove that $\phi $ separates $F_{7}$
from the oppressive sets of $F_{7}$ in $F_{3},F_{4}$.

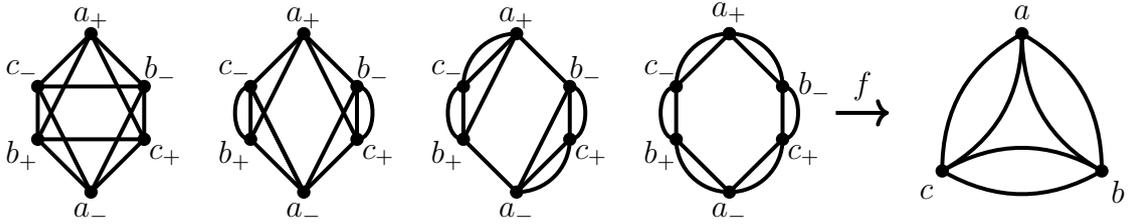
\begin{figure}[ht]
\centering
\begin{tikzpicture}[line width=1.5pt, scale = 0.7] 

\begin{scope}[decoration={
    markings,
    mark=at position .6 with {\arrow{}}}]
   \draw[->] (2,-1.5) --  (3,-1.5);
  \draw[postaction={decorate}] (5.5,-0) to [out=-30,in=90] (7,-2.6);
 \draw[postaction={decorate}] (5.5,0) to [out=-90,in=150] (7,-2.6);
   \draw[postaction={decorate}] (5.5,0) to [out=210,in=90]  (4,-2.6);
   \draw[postaction={decorate}] (5.5,0) to [out=-90,in=30]  (4,-2.6);
   \draw[postaction={decorate}] (4,-2.6) to [out= 30,in= 150] (7,-2.6);
   \draw[postaction={decorate}] (4,-2.6) to [out= -30,in= 210]  (7,-2.6);
  
\filldraw (0,0) circle (2.5pt);
 \filldraw       (5.5,0) circle (2.5pt);
 \filldraw       (4,-2.6) circle (2.5pt);
 \filldraw       (7,-2.6) circle (2.5pt);

\draw (5.5, 0.4)  node[text=black, scale=1]{$a$};
\draw (7+0.3, -3)  node[text=black, scale=1]{$b$};
 \draw (4- 0.3, -3)  node[text=black, scale=1]{$c$};
\draw (2.5,-1) node[text=black, scale=1]{$f$};
\end{scope}

\begin{scope}[decoration={
    markings,
    mark=at position .6 with {\arrow{}}}]
   \draw[postaction={decorate}] (-1,-1) to  (0,0);
  \draw[postaction={decorate}] (-1,-1) to [out=90,in=180] (0,0);
\draw[postaction={decorate}] (1,-1) to   (0,0);
   \draw[postaction={decorate}] (1,-1) to [out=90,in=0]  (0,0);
   \draw[postaction={decorate}] (1,-1) to  (1,-2);
   \draw[postaction={decorate}] (1,-1) to [out= 0,in= 0]  (1,-2);
   \draw[postaction={decorate}] (0,-3) to  (1,-2);
      \draw[postaction={decorate}] (0,-3) to [out=0,in=-90]  (1,-2);
   \draw[postaction={decorate}] (-1,-1) to  (-1,-2);
      \draw[postaction={decorate}] (-1,-1) to [out=180,in=180]  (-1,-2);
   \draw[postaction={decorate}] (0,-3) to  (-1,-2);
      \draw[postaction={decorate}] (0,-3) to [out=180,in=-90]  (-1,-2);  
\filldraw (0,0) circle (2.5pt);
 \filldraw       (1,-1) circle (2.5pt);
 \filldraw       (1,-2) circle (2.5pt);
 \filldraw       (0,-3) circle (2.5pt);
 \filldraw       (-1,-2) circle (2.5pt);
 \filldraw       (-1,-1) circle (2.5pt);

\draw (0, 0.4)  node[text=black, scale=1]{$a_+$};
\draw (1+0.6, -1)  node[text=black, scale=1]{ $b_{-}$};
 \draw (1+ 0.4, -2- 0.3)  node[text=black, scale=1]{$c_+$};
\draw (0, -3- 0.4)  node[text=black, scale=1]{$a_{-}$};
\draw (-1 - 0.3, -2- 0.3)  node[text=black, scale=1]{$b_+$};
\draw (-1 - 0.3, -1+0.3)  node[text=black, scale=1]{$c_{-}$};

\end{scope}
\begin{scope}[decoration={
    markings,
    mark=at position .6 with {\arrow{}}}]
   \draw[postaction={decorate}] (-4+-1,-1) to  (-4+0,0);
  \draw[postaction={decorate}] (-4+-1,-1) to [out=90,in=180] (-4+0,0);
\draw[postaction={decorate}] (-4+1,-1) to   (-4+0,0);
   \draw[postaction={decorate}] (-4-1,-2) to  (-4+0,0);
   \draw[postaction={decorate}] (-4+1,-1) to  (-4+1,-2);
   \draw[postaction={decorate}] (-4+1,-1) to [out= 0,in= 0]  (-4+1,-2);
   \draw[postaction={decorate}] (-4+0,-3) to  (-4+1,-2);
      \draw[postaction={decorate}] (-4+0,-3) to [out=0,in=-90]  (-4+1,-2);
   \draw[postaction={decorate}] (-4+-1,-1) to  (-4+-1,-2);
      \draw[postaction={decorate}] (-4+-1,-1) to [out=180,in=180]  (-4+-1,-2);
   \draw[postaction={decorate}] (-4+0,-3) to  (-4+-1,-2);
      \draw[postaction={decorate}] (-4+0,-3) to  (-4+1,-1);  
\filldraw (-4+0,0) circle (2.5pt);
 \filldraw       (-4+1,-1) circle (2.5pt);
 \filldraw       (-4+1,-2) circle (2.5pt);
 \filldraw       (-4+0,-3) circle (2.5pt);
 \filldraw       (-4+-1,-2) circle (2.5pt);
 \filldraw       (-4+-1,-1) circle (2.5pt);

\draw (-4+0, 0.3)  node[text=black, scale=1]{$a_+$};
\draw (-4+1+0.3, -1+0.3)  node[text=black, scale=1]{$b_{-}$};
 \draw (-4+1+ 0.4, -2- 0.3)  node[text=black, scale=1]{$c_+$};
\draw (-4+0, -3- 0.4)  node[text=black, scale=1]{$a_{-}$};
\draw (-4+-1 - 0.3, -2- 0.3)  node[text=black, scale=1]{$b_+$};
\draw (-4+-1 - 0.3, -1+0.3)  node[text=black, scale=1]{$c_{-}$};

\end{scope}

\begin{scope}[decoration={
    markings,
    mark=at position .6 with {\arrow{}}}]
   \draw[postaction={decorate}] (-8+-1,-1) to  (-8+0,0);
  \draw[postaction={decorate}] (-8+-1,-1) to  (-8+0,-3);
\draw[postaction={decorate}] (-8+1,-1) to   (-8+0,0);
   \draw[postaction={decorate}] (-8+1,-1) to  (-8+0,-3);
   \draw[postaction={decorate}] (-8+1,-1) to  (-8+1,-2);
   \draw[postaction={decorate}] (-8+1,-1) to [out= 0,in= 0]  (-8+1,-2);
   \draw[postaction={decorate}] (-8+0,-3) to  (-8+1,-2);
    \draw[postaction={decorate}] (-8+0,-3) to   (-8-1,-1);
   \draw[postaction={decorate}] (-8+-1,-1) to  (-8+-1,-2);
      \draw[postaction={decorate}] (-8+-1,-1) to [out=180,in=180]  (-8+-1,-2);
   \draw[postaction={decorate}] (-8+0,-3) to  (-8+-1,-2);
    \draw[postaction={decorate}] (-8+0,0) to   (-8+-1,-2);  
    \draw[postaction={decorate}] (-8+0,0) to   (-8+1,-2);  
\filldraw (-8+0,0) circle (2.5pt);
 \filldraw       (-8+1,-1) circle (2.5pt);
 \filldraw       (-8+1,-2) circle (2.5pt);
 \filldraw       (-8+0,-3) circle (2.5pt);
 \filldraw       (-8+-1,-2) circle (2.5pt);
 \filldraw       (-8+-1,-1) circle (2.5pt);

\draw (-8+0, 0.3)  node[text=black, scale=1]{$a_+$};
\draw (-8+1+0.3, -1+0.3)  node[text=black, scale=1]{$b_{-}$};
 \draw (-8+1+ 0.4, -2- 0.3)  node[text=black, scale=1]{$c_+$};
\draw (-8+0, -3- 0.4)  node[text=black, scale=1]{$a_{-}$};
\draw (-8+-1 - 0.3, -2- 0.3)  node[text=black, scale=1]{$b_+$};
\draw (-8+-1 - 0.3, -1+0.3)  node[text=black, scale=1]{$c_{-}$};

\end{scope}

\begin{scope}[decoration={
    markings,
    mark=at position .6 with {\arrow{}}}]
   \draw[postaction={decorate}] (-12+-1,-1) to  (-12+0,0);
  \draw[postaction={decorate}] (-12-1,-1) to  (-12+1,-1);
\draw[postaction={decorate}] (-12+1,-1) to   (-12+0,0);
   \draw[postaction={decorate}] (-12-1,-1) to   (-12+0,-3);
   \draw[postaction={decorate}] (-12+1,-1) to  (-12+1,-2);
   \draw[postaction={decorate}] (-12+1,-1) to   (-12+0,-3);
   \draw[postaction={decorate}] (-12+0,-3) to  (-12+1,-2);
      \draw[postaction={decorate}] (-12+0,0) to  (-12+1,-2);
   \draw[postaction={decorate}] (-12+-1,-1) to  (-12+-1,-2);
      \draw[postaction={decorate}] (-12+1,-2) to  (-12+-1,-2);
   \draw[postaction={decorate}] (-12+0,-3) to  (-12+-1,-2);
      \draw[postaction={decorate}] (-12+0,0) to   (-12-1,-2);  
\filldraw (-12+0,0) circle (2.5pt);
 \filldraw       (-12+1,-1) circle (2.5pt);
 \filldraw       (-12+1,-2) circle (2.5pt);
 \filldraw       (-12+0,-3) circle (2.5pt);
 \filldraw       (-12+-1,-2) circle (2.5pt);
 \filldraw       (-12+-1,-1) circle (2.5pt);

\draw (-12+0, 0.3)  node[text=black, scale=1]{$a_+$};
\draw (-12+1+0.3, -1+0.3)  node[text=black, scale=1]{$b_{-}$};
 \draw (-12+1+ 0.4, -2- 0.3)  node[text=black, scale=1]{$c_+$};
\draw (-12+0, -3- 0.4)  node[text=black, scale=1]{$a_{-}$};
\draw (-12+-1 - 0.3, -2- 0.3)  node[text=black, scale=1]{$b_+$};
\draw (-12+-1 - 0.3, -1+0.3)  node[text=black, scale=1]{$c_{-}$};

\end{scope}

\end{tikzpicture}
\caption{The graph $X_{1/4}$ if (1) all $M,N,P$ are odd, (2) only $N$ is even, (3) only $M$ is odd, (4) all $M,N,P$ are even. In all cases, $X_{1/4}\rightarrow X_{1/2}$ is a double covering.}
\label{case1}
\end{figure}

We check first  $\phi |_{F_{4}}$ separates $F_{7}$ from the oppressive set $%
A_{\rho _{2}}= A_{f}.$ Fixing a based point $x_0 = a \in X_{1/2}$, the  oppressive set $A_{\rho _{2}}$
consists of all $g\in \pi _{1}(X_{1/2},x_{0})$
representing by a cycle $\gamma $ in $X_{1/2}$ such that $\gamma =\rho_2 (\mu
_{1})\cdot \rho_2 (\mu _{2}),$ a concatenation, where $\mu _{1}$ is a
non-trivial simple non-closed path in $X_{1/4}$ going from the base point $y_{0}=a^{+}$ to some vertex $%
y_{1},$ and $\mu _{2}$ is either trivial or a simple non-closed path in $X_{1/4}$
going from some vertex $y_{2}$ to $y_{0},$ with $y_{1}\neq y_{2}\neq y_{0}.$
Note that each element in the oppressive set $%
A_{\rho _{2}}$ is conjugated to an element represented by the image of a simple path
connecting $a_{+},a_{-},$ or $b_{+},b_{-},$ or $c_{+},c_{-}.$ For any
element $g\in A_{\rho _{2}},$ it is not hard to see that $\phi (g)\notin
\phi (F_{7})$ if and only if $\phi (h_{1}gh_{2})\notin \phi (F_{7})$ for any 
$h_{1},h_{2}\in F_{7}.$ Since $F_{7}$ is of index 2 (and thus normal) in $F_{4},$ it suffices
to prove that there exists  one (and thus arbitrarily many) path connecting $%
a_{+},a_{-},$ or $b_{+},b_{-},$ or $c_{+},c_{-}$ respectively, whose image
does not lie in $\phi (F_{7})$.  We know that the subgroup $ F_3 = \pi_1({X_0})$ is generated by $ab,bc,ca$, whose images are orientation-preserving in the Coxeter group $\Delta_{MNP}$ (see \cite[Section 3.2]{Jan21}). In order to generate $\Delta_{MNP}$, the subgroup $F_4$ must have an orientation-reversing element $g$. Write $F_4 = F_7 \cup g F_7$. Note that an element of the normal subgroup $F_7$ in $F_4$ is represented by a closed loop based at $a$, whose lift will again be a loop  since $f$ is a two-fold covering.  Thus the image of simple paths connecting $a_{+},a_{-},$ or $b_{+},b_{-},$ or $c_{+},c_{-}$ does not lie in $F_7$.  From this, we see that $A_{\rho _{2}}$ is separated from $%
F_{7}$ by $\phi |_{F_{4}}.$

We check that $\phi |_{F_{3}}$ separates $F_{7}$ from the oppressive set $%
A_{\rho _{1}}.$ Let%
\begin{equation*}
K=\langle x,y,z:x^{M}=1, y^{N}= 1, z^{P}=1, xyz=1\rangle
\end{equation*}%
and $X_K$ be the presentation complex of $K.$ Note that $K$ is the von Dyck
group, the index-2 subgroup of $\Delta_{MNP} $ and $\phi (F_{3})=K.$ Let $\tilde{X}_K
$ be the universal cover of $X_K$. Identify the 2-cells in $\tilde{X}_K$ with
the same boundary (i.e.~$M$ copies of the 2-cells whose boundary word is $%
x^{M}$ are collapsed to a single 2-cell, and similarly for $y^{M},z^{P}$).
Denote by $X^{\prime }$ the resulting complex. Note that $X_K^{\prime }$
admits a metric so that it is isometric to $\mathbb{H}^{2}$ (or $\mathbb{R}%
^{2}$ when $M=N=P=3$). We proceed the proof in two subcases.

\begin{enumerate}[label= Subcase 1.\arabic*.]
\item $M=2m+1,N=2n+1,P=2p+1\geq 3$ are odd. The injection $\rho _{1}:F_{7}\rightarrow F_{3}$ is induced by the graph
immersion in the first row of Figure \ref{pic-333-great}. Attach now four 2-cells (with one along the great cycle and the other three ones along the monochrome cycles) to the graph $X_{1/4}$ of $F_{7}$ as shown in Figure \ref{pic-7-case1}
\begin{figure}[ht]
\centering
\includegraphics[width=0.5\textwidth]{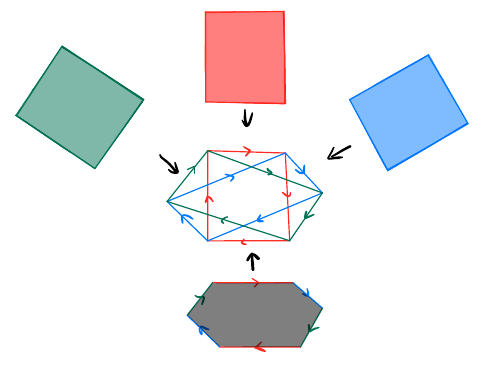}
\caption{Attaching four $2$-cells to the graph $%
{X}_{1/4}$.}
\label{pic-7-case1}
\end{figure}
to get a 2-complex $Y.$ Note that the boundary of each of the 2-cells is
null-homotopic in  $X_K$. By shrinking the corresponding edges in $Y$, we obtain a 2-complex $\bar{Y}$ which contains $\bar{X}_{1/4}$ as a subcomplex.  Therefore, there is
an induced map $f:Y\rightarrow \bar{Y} \rightarrow X_K.$ The group homomorphism $\phi
|_{F_{7}}:F_{7}\hookrightarrow F_{3}\rightarrow K$ factors through $\pi
_{1}(Y)\rightarrow \pi _{1}(X_K)=K.$ Note that we can put a metric on $\bar{Y}$ by first putting a metric on the $2$-cells using the corresponding metric on $X'_K$.  This allows us to conclude the composite $\tilde{\bar{Y}}%
\rightarrow \tilde{X}_K\rightarrow X_K^{\prime }$ is local isometric embedding, where $\tilde{\bar{Y}}$  is the universal cover. In fact, restricted to each $2$-cell, it is already an isometry. One just needs to check at every vertex, it is a local isometric embedding for which one can check through the induced map on the link of each vertex.   Now since $X_K^{\prime }$  is  CAT(0), the local embedding $\tilde{\bar{Y}}\rightarrow X_K^{\prime }$ is in fact a global embedding (\cite[II.4.14]{BH99}). By Lemma \ref{lemma-imy-sep}, we
see that $\phi |_{F_{3}}$ separates $F_{7}$ from the oppressive set $A_{\rho
_{1}}.$

\item  One of $M,N,P$ is even. In this case, the injection $F_{7}\rightarrow F_{3}$ is induced by the graph maps  in Figure \ref{pic-333-new} (comparing with the last $3$ rows in Figure \ref{pic-333-great},  more labels have been added to clarify the situation.)
\begin{figure}[ht]
\centering
\includegraphics[width=0.6\textwidth]{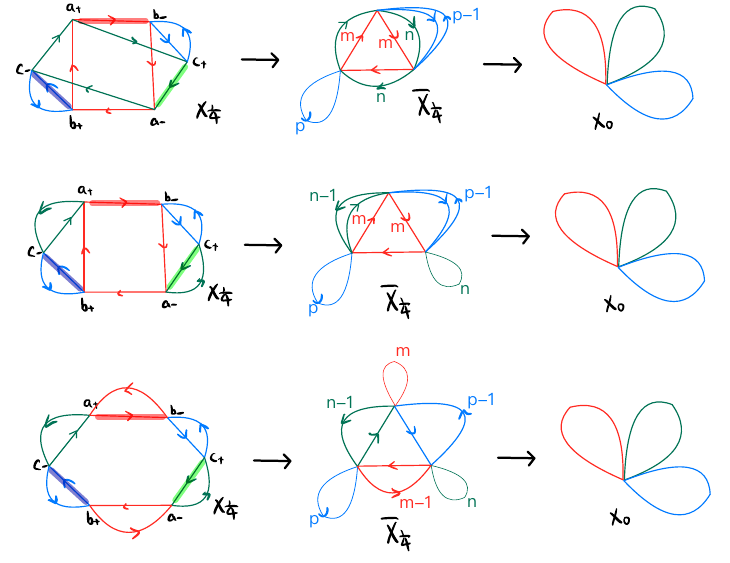}
\caption{The  map $\phi: X_{1/4} \rightarrow \Bar{X}_{1/4} \rightarrow X_0 $, when $(M,N,P) = (2m+1,2n+1,2p)$, $(2m+1,2n,2p)$ or $(2m,2n,2p)$.}
\label{pic-333-new}
\end{figure}

For each even integer of $\{M,N,P\},$ attach two 2-cells to the graph $\bar{X%
}_{1/4}$ of $F_{7}$ along the corresponding two simple monochrome cycles
respectively. The degree of the boundary maps is 2. And for each odd integer of  $\{M,N,P\}$, we attach a $2$-cell to the corresponding simple monochrome cycle with boundary map degree $1$.  Attach an additional
2-cell to the great cycle labeled $xyz$. Note that here in Figure \ref{pic-333-new}, the $x$ (resp. $y,z$) edges are labeled by red (resp. green, blue). The resulting complex $\Bar{Y}$ is obtained from the complexes in Figure \ref{pic-2-cells} by shrinking thickened edges and attaching 2-cells along the great cycles labeled by $xyz$, respectively.

\begin{figure}[ht]
\centering
\includegraphics[width=0.7\textwidth]{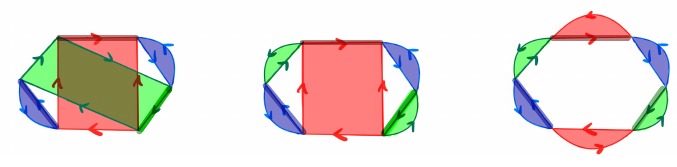}
\caption{The complex $\Bar{Y}$ is obtained from the above complexes by shrinking thickened edges and attaching 2-cells along the great cycles  when one, two, or three of $\{M,N,P\}$ is even.}
\label{pic-2-cells}
\end{figure}

Note that $\pi _{1}(\Bar{Y})$ is a free product $\mathbb{Z}/2\ast \mathbb{Z}/2\ast
F_{2},$ $\mathbb{Z}/2\ast \mathbb{Z}/2\ast \mathbb{Z}/2\ast \mathbb{Z}/2\ast 
\mathbb{Z},$ or a free product of 6 copies of $\mathbb{Z}/2,$ depending on
the number even numbers in $\{M,N,P\}$. Let $%
Y^{\prime }$ be the complex obtained from the universal cover of $\Bar{Y}$ by
 identifying all the 2-cells with the same boundary. Recall that the complex $X_K^{\prime }$ admits a metric so that it
is isometric to the CAT(0) space $\mathbb{H}^{2}$. Note that each attached 2-cell
is null-homotopic in the presentation complex $X_K$ of $K.$ In fact, we could pull the metric on the $2$-cells of $X_K'$ and get a metric on $\bar{Y}$ and $Y'$. Just as before, one sees that the induced map $%
f: \Bar{Y}\rightarrow X_K,f^{\prime }:Y^{\prime }\rightarrow X_K^{\prime }$ are both
locally isometric embedding (under the induced metric). 
Since $X_K^{\prime }$ is isometric to the CAT(0) space $\mathbb{H}^{2}$, the
local embedding $Y^{\prime }\rightarrow X_K^{\prime }$ is a global embedding.
Therefore, $\tilde{\Bar{Y}}\rightarrow \tilde{X}_K$ is an embedding. By   
Lemma \ref{lemma-imy-sep}, we see that $\phi |_{F_{3}}$ separates $F_{7}$ from the
oppressive set $A_{\rho _{1}}.$
\end{enumerate}

\item Suppose that $P=2,M\geq 4,N\geq 4.$ When both $M=2m,N=2n$ are even, the
Artin group of $\Gamma $ splits an HNN extension $A\ast _{B,\beta }$ with $
A=\langle x,y\rangle \cong F_{2},B=\langle x^{M/2},y^{N/2},x^{-1}y\rangle $
and $\beta (x^{M/2})=x^{M/2},\beta (y^{N/2})=y^{N/2},\beta
(x^{-1}y)=yx^{-1}. $ The splitting is induced by immersion of graphs in Figure \ref{2-e-e}. Let 


\begin{equation*}
K=\langle x,y:x^{M}= 1, y^{N}=1, (x^{-1}y)^{2}=1\rangle
\end{equation*}%
and $X_K$ be the presentation complex of $K$. Note that $K$ is a von Dyck group but with a nonstandard presentation. We have to modify the presentation here to apply Lemma \ref{lemma-imy-sep} since $\pi_1(X_A)$ is only a free group of rank $2$, where $X_A$ is as in Figure \ref{2-e-e}.  Let $\tilde{X}_K$ be the universal
cover of $X_K.$ Identify all the 2-cells in $\tilde{X}_K$ with the same boundaries
(i.e.~the $M$ copies of the 2-cell whose boundary word is $x^{M}$ are collapsed
to a single 2-cell, and similarly for $y^{N},(xy)^{2}$). Denote by $X^{\prime
}_K $ the resulting complex. Note that $X^{\prime }_K$ admits a metric so that
it is isometric to a CAT(0) space $\mathbb{R}^{2}$ or $\mathbb{H}^{2}.$
Attach two 2-cells to $\bar{X}_{B}$ along the two simple monochrome cycles (see Figure \ref{2-e-e})
and one $2$-cell to the inner circle labeled by $x^{-1}y,$ all with boundary maps of degree 2. Denote the resulting complex by $Y.$
Let $Y^{\prime }$ be the complex obtained from the universal cover $\tilde{Y}
$ by identifying 2-cells with the same boundary $x^{M},y^{N},(x^{-1}y)^2$. Note that
each attached 2-cell is null-homotopic in the presentation complex $X_K$ of $K.
$ A similar argument shows that $\tilde{Y}\rightarrow \tilde{X}_K$ is an
embedding and thus $\phi $ separates $B$ ($\beta (B),$ respectively) from the
oppressive set $A_{\rho_1}$ (resp. $A_{\rho_2}$)  by Lemma \ref{lemma-imy-sep}.

When at most one of $M,N$ is even, the Artin group $\mathrm{Art}_{2MN}$ splits as a
free product with amalgamation $A\ast _{C}B$ where $A=F_{2},B=F_{3},C=F_{5}$
and the injections $C\hookrightarrow A,C\hookrightarrow B$ are induced by
immersion of graphs, see Theorem \ref{thm-split-sum} (ii). For the top case in Figure \ref{pic-2-odd-odd}, attach four 2-cells to $\bar{X}_{C}$ , with two of them along the two simple monochrome
cycles with boundary maps of degree 1, and the remaining two along the two circles labeled by $%
x^{-1}y$ with boundary maps of degree 2. For the bottom case in Figure \ref{pic-2-odd-odd}, attach five 2-cells to $\bar{X%
}_{C}$, with one along the simple blue cycle with a boundary
map of degree 1, two along the two red cycles with boundary maps of degree 2, and the remaining two along the inner
cycles labeled by $x^{-1}y,$ with boundary maps of degree 2. A similar
argument as before using Lemma \ref{lemma-imy-sep} finishes the proof for $\rho_1$. For $\rho_2$, one argues the same as Case 1 for the map $\rho_2$. 

\item Suppose now that $M=2,N=3,P\geq 6$ is even, Theorem \ref{23even}
implies that $$\mathrm{Art}_{2,3,2m}=\langle
a,b,c:ac=ca,bcb=cbc,(ab)^{m}=(ba)^{m}\rangle $$ ($m\geq 3$) is isomorphic to $
F_{3}\ast _{F_{7}}F_{4}.$ The generators of $F_{3}$ are given by $%
x=c(ba)c^{-1},y=bc,\delta =bc(ba)^{m}.$ The injections $\rho_1: F_{7}\hookrightarrow
F_{3},\rho_2: F_{7}\hookrightarrow F_{4}$ are induced by immersions of graphs in Figure \ref{pic-23even}. Note that the second injection is induced by a $2$-fold covering. In the picture, this is induced by rotating the graph $X_{\frac{1}{4}}$ for $180$ degrees and the quotient graph is $X_{\frac{1}{2}}$.

We proceed to show first that $\phi $ separates $F_{7}$ from the oppressive set $A_{\rho_1 }$. Let 
\begin{equation*}
D_{2,3,2m}=\langle u,v:u^{2m}=v^{3}=(uv)^{2}=1\rangle
\end{equation*}%
be the corresponding von Dyck group. Viewing $D_{2,3,2m}$ as a subgroup of $\Delta_{MNP} ,$ we have $%
u=ab,v=bc.$ The group $F_{3}=\langle x,y,\delta \rangle $ is mapped by $\phi 
$ onto $D_{2,3,2m}$ with 
\begin{equation*}
\phi (x)=v^{-1}uv,\phi (y)=v,\phi (\delta )=vu^{-m}.
\end{equation*}%
We want to apply Lemma \ref{lemma-imy-sep} now, but it only allows us to attach $2$-cells to the graph $X_0$ which has three $1$-cells. So we modify the presentation of $D_{2,3,2m}$ first. Note that the map $\phi |_{F_{3}}:F_{3}\rightarrow D_{2,3,2m}$ factors through 
\begin{equation*}
\bar{\phi}:K:=\langle x,y,\delta :x^{2m}=1, y^{3}=1, (yx)^{2}=1,\delta y x^m y
=1\rangle \rightarrow D_{2,3,2m},
\end{equation*}%
which is an isomorphism with the inverse $\psi :D_{2,3,2m}\rightarrow K$ given by%
\begin{equation*}
\psi (u)=yxy^{-1},\psi (v)=y.
\end{equation*}%
Note that $K$ is isomorphic to the group given by the subpresentation 
\begin{equation*}
K_{1}=\langle x,y:x^{2m}=1, y^{3}=1,(yx)^{2}=1\rangle ,
\end{equation*}%
thus it is still isomorphic to the von Dyck group. 

We now put a  metric on the presentation complex of $D_{2,3,2m}$ just as we have done in the previous cases. In fact, denote the presentation complex of $D_{2,3,2m}$ by $X_D$, taking its universal cover $\tilde{X}_{D}$. There are many $2$-cells that share the same boundary, for example, there are $2m$ $2$-cells that share the same boundary corresponding to the relator $u^{2m}=1$. We collapse all the $2$-cells sharing the same boundary to obtain a contractible $2$-complex which we can put a hyperbolic metric so that it is isometric to the hyperbolic space (or isometric to the Euclidean space, if $m=3$). This metric then descends to a metric on $X_D$ in such a way that its $1$-cells are piecewise geodesics, and all the $2$-cells carry a hyperbolic (resp. Euclidean) metric, see Figure \ref{pic-238} for the case $m=4$. We denote this metric space by $X_D'$.

\begin{figure}[ht!]
\centering
\includegraphics[width=70mm]{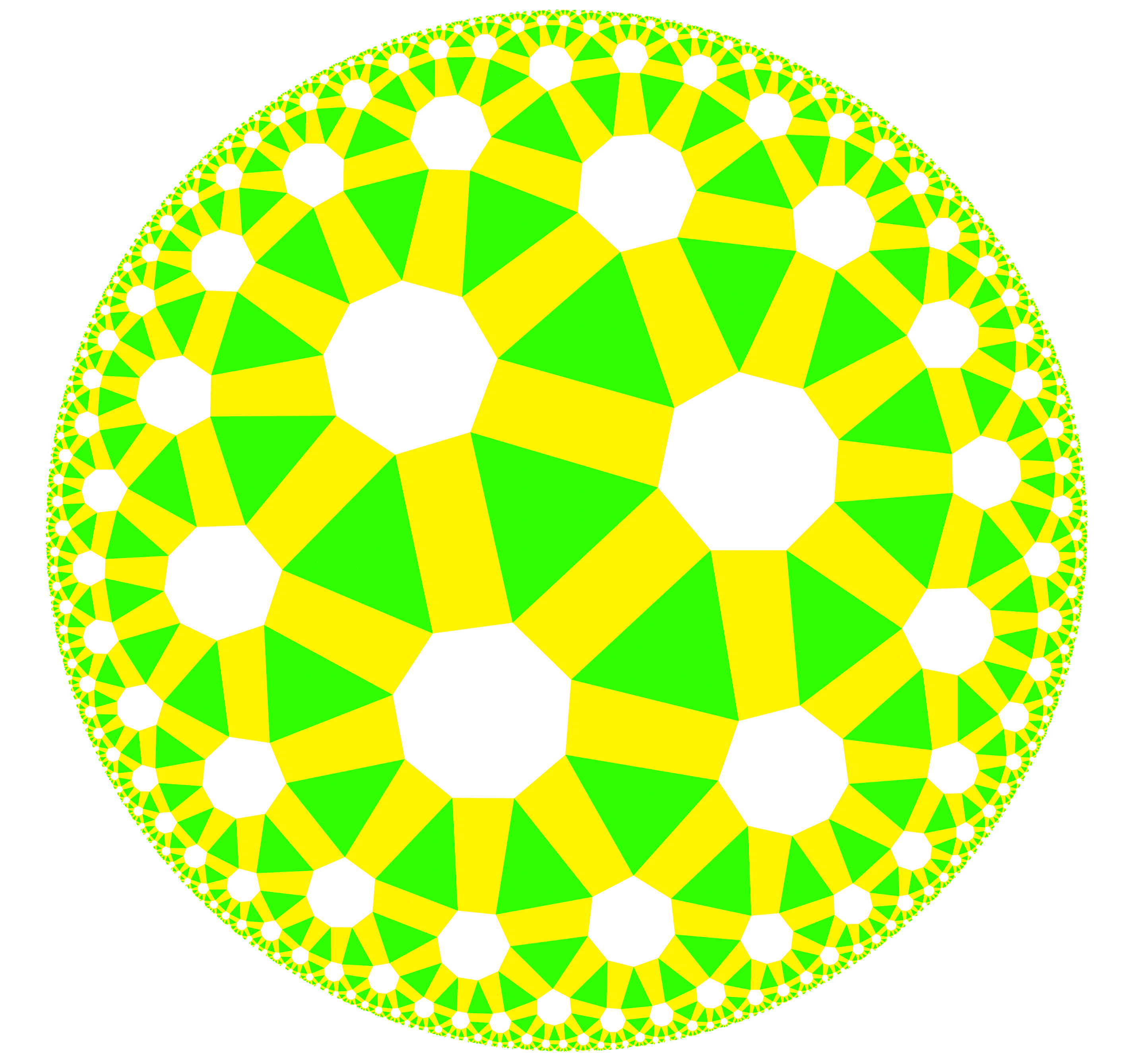}
\caption{Tessellation of the von Dyck group $D_{2,3,8}$. }
\label{pic-238}
\end{figure}

We then add an additional $2$-cell to $X_D$  corresponding to the relator $\delta y x^m y
=1$ obtaining the presentation complex $X_K$ of $K$. To put a metric on $X_K$, we will  construct the metric on the additional $2$-cell using metric on $X_D'$ which then descends to metric on $X_K$. Pick a vertex in the $1$-skeleton of $X_D'$, and we get a piecewise geodesic by traveling  along the $1$-skeleton according to the word $yx^my$. We then glue $2m+3$ rectangles to $X_D'$ to make it a CAT(0) space, as follows. In fact, for each letter $x$ (resp. $y$) in $yx^my$ we first attach a rectangle of the form $[0,|x|] \times [0,1]$ (resp. $[0,|y|] \times [0,1]$) where $|x|$ (resp. $|y|$) denotes the length of the geodesic corresponding to the word $x$ (resp. $y$), to $X_D'$ by identifying the edge $[0,|x|] \times \{0\}$ (resp.  $[0,|y|] \times \{0\}$ ) with the geodesic corresponding to the letter $x$ (resp. $y$). To ensure that it is CAT(0), we have to ensure the link at each break point has a length of at least $2\pi$ (see for example \cite[II.5.6]{BH99}). For that, at each break point $z_i$ (say it is the end point of the geodesic $Z_i$ and the starting point of the geodesic $Z_{i+1}$), we further attach a rectangle $[-1,1]\times [0,1]$ to the metric space, by identifying $[-1,0] \times \{0\}$ with $z_i\times [0,1] \subset Z_{i}\times [0,1] $ via an orientation reversing isometry and $[0,1] \times \{0\}$ with $z_i\times [0,1] \subset Z_{i+1}\times [0,1]$ via an orientation preserving isometry. We denote the result CAT(0) complex by $X_K'$. It also descends to a metric on $X_K$ as well as its universal cover $\tilde{X}_K$.

Attach now seven 2-cells to $\bar{X}_{1/4}$ (see Figure \ref{pic-23even}), one along the simple cycle in the middle labeled by $
x^{m}$ with boundary maps of degree 2, one along the simple cycle labeled by 
$y^{3}$ with a boundary map of degree 1, two along the simple cycle labeled by 
$yx$ with boundary map of degree 2, three along $\delta y x^m y
$.
Let $Y$ be the resulting complex and $Y^{\prime }$ be the complex obtained
from the universal cover $\tilde{Y}$ by identifying 2-cells with the same
boundaries labeled by $x^{m},yx$. On each $2$-cell in these complexes, we also put a metric on it using the metric of the corresponding $2$-cell in $X_K'$. This puts a metric on all the three complexes $Y, \tilde{Y}, Y'$. By construction, the image of the boundary of each attached 2-cell in $Y$ is null-homotopic in the
presentation complex $X_K$, we have a naturally induced map $f:Y\rightarrow X_K, f^{\prime
}:Y^{\prime }\rightarrow X_K^{\prime }$. One checks that now that  both maps are  locally isometric embedding. In fact, restricted to each $2$-cell, it is already an isometry. One just needs to check at every vertex, it is an isometric embedding that one can check through the induced map on the link of each vertex.
Now since $X_K^{\prime }$  is  CAT(0), the local embedding $Y^{\prime
}\rightarrow X_K^{\prime }$ is a global embedding (\cite[II.4.14]{BH99}). Therefore, $\tilde{Y}%
\rightarrow \tilde{X}_K$ must also be an embedding. Now again by  Lemma \ref{lemma-imy-sep}, we see
that $\phi $ separates $F_{7}$ from the oppressive set $A_{\rho_1 }.$ On the other hand, the same proof in  case (1) for dealing with $\rho_2$ also implies  $\phi |_{F_{4}}$ separates $F_{7}$ from the
oppressive set $A_{\rho _{2}}$ in this case. 
\end{enumerate}

\end{proof}

\begin{proof}[Proof of Theorem \protect\ref{th3}]\label{proof:thm3}
It's enough to prove that the triangle Artin groups, not covered in Theorem \ref{mthm-poly-free}. But the remaining cases are either of irreducible finite type or a direct product of a dihedral Artin group with an infinite cyclic group which are known to be virtually poly-free. In fact, when it is not a direct product, it is either of the form $(2,3,3)$, $(2,3,4)$ or $(2,3,5)$. In the $(2,3,3)$ case, it is a braid group that is known to be virtually normally poly-free \cite{FN62}.  The Artin group of $(2,3,4)$ can be embedded into braid groups \cite{Cr99} hence it is also virtually normally poly-free.  The fact that the Artin group of type $(2,3,5)$ is virtually poly-free can be deduced from  the fibration in \cite[Proposition 2]{Br73}. Since any dihedral Artin group is normally poly-free \cite[Lemma 3.4]{Wu22}, it follows that any Artin group of type $(2,2,n)$ is also normally poly-free. This finishes the proof. 
\end{proof}

\section{Commutator subgroups of  triangle Artin groups}

In this section, we use the Reidemeister--Schreier method to calculate the first homology  of the commutator subgroups of many triangle Artin groups. As an application,  we show that those triangle Artin groups can not be poly-free.

\subsection{Presentations of the commutator subgroups}

We consider presentations of the commutator subgroups of the Artin group$$%
\mathrm{Art}_{MNP}=\langle a,b,c \mid \{a,b\}_{M} \cdot\{b,a\}_M^{-1} =1,\{b,c\}_{N}\cdot\{c,b\}_N^{-1} =1,\{a,c\}_{P}\cdot \{c,a\}_P^{-1} =1\rangle.$$ Let $%
f:\mathrm{Art}_{MNP}\rightarrow \mathbb{Z}$ be the degree map that maps each generator to 
$1\in \mathbb{Z}.$ When the set $\{M,N,P\}$ has at least two odd numbers, $\ker f$ is the commutator subgroup $\mathrm{Art}_{MNP}^{\prime }$ (see \cite{MR}, Proposition 3.1). We choose $%
R=\{a^{i}:i\in \mathbb{Z}\}$ as a Schreier system for $\mathrm{Art}_{MNP}$ modulo $\mathrm{Art}_{MNP}^{\prime }$.
The Reidemeister theorem (cf. Theorem \ref{schreier}) implies that $\mathrm{Art}_{MNP}^{\prime }$ is generated by $$%
\{s_{a^{i},b},s_{a^{i},c}, i\in \mathbb{Z}\}$$
since $s_{a^{i},a}=1$.
\begin{lemma}
\label{relator}
\begin{enumerate}
    \item If $M=2,$ we have $$s_{a^{i+1},b}=s_{a^{i},b}$$ for each $%
i\in \mathbb{Z}$;

    \item If $N=2k+1$, so there is a relator $(bc)^{k}b(c^{-1}b^{-1})^{k}c^{-1}$ in $\mathrm{Art}_{MNP}$, we have the corresponding relation in $\ker f$:
\begin{equation*}
(\Pi _{i=0}^{k-1}(s_{a^{2i},b} )(s_{a^{2i+1},c}))(s_{a^{2k},b})=(s_{0,c})(\Pi
_{i=1}^{k}(s_{a^{2i-1},b})(s_{a^{2i},c}));
\end{equation*}

    \item If $N=2m$, so there is a relator $(bc)^{m}(b^{-1}c^{-1})^{m}$ in $\mathrm{Art}_{MNP}$, we have the corresponding relation in $\ker f$:
\begin{equation*}
\Pi _{i=0}^{m-1}(s_{a^{2i},b})(s_{a^{2i+1},c})=\Pi
_{i=0}^{m-1}(s_{a^{2i},c})(s_{a^{2i+1},b}).
\end{equation*}

    \item If $P=2l+1$, so there is a relator $(ac)^{l}a(c^{-1}a^{-1})^{l}c^{-1}$ in $\mathrm{Art}_{MNP}$, we have the corresponding relation in $\ker f$:

\begin{equation*}
\Pi _{i=0}^{l-1}(s_{a^{2i+1},c})=\Pi _{i=0}^{l}(s_{a^{2i},c}).
\end{equation*}
    
\end{enumerate}

\end{lemma}

\begin{proof}
Note that $$1=\tau
(aba^{-1}b^{-1})=(s_{a^{0},a})(s_{a^{1},b})(s_{a^1,a}^{-1})(s_{a^{0},b}^{-1})=(s_{a^{1},b})(s_{a^{0},b}^{-1}). 
$$ By Lemma \ref{shift}, we have $$
1=a^{i}(s_{a^{1},b})(s_{a^{0},b}^{-1})a^{-i}=(s_{a^{i+1},b})(s_{a^{i},b}^{-1}),$$
hence $s_{a^{i+1},b}=s_{a^{i},b}$ for each $i\in \mathbb{Z}$.

Similarly, when $N=2k+1,$ we have%
\begin{eqnarray*}
1 &=&\tau ((bc)^{k}b(c^{-1}b^{-1})^{k}c^{-1}) \\
&=&(s_{a^{0},b})(s_{a^{1},c})(s_{a^{2},b})...(s_{a^{2k-1},c})(s_{a^{2k},b})(s_{a^{2k},c}^{-1})(s_{a^{2k-1},b}^{-1})\cdots (s_{0,c}^{-1})
\\
&=&(\Pi _{i=0}^{k-1}(s_{a^{2i},b})(s_{a^{2i+1},c}))(s_{a^{2k},b})(\Pi
_{i=k}^{1}(s_{a^{2i},c}^{-1})(s_{a^{2i-1},b}^{-1}))(s_{0,c}^{-1}),
\end{eqnarray*}%
implying that $$(\Pi
_{i=0}^{k-1}(s_{a^{2i},b})(s_{a^{2i+1},c}))(s_{a^{2k},b})=(s_{0,c})(\Pi
_{i=1}^{k}(s_{a^{2i-1},b})(s_{a^{2i},c})).$$ 
Similarly for $(bc)^{m}(b^{-1}c^{-1})^{m}$ , we have%
\begin{eqnarray*}
1 &=&\tau ((bc)^{m}(b^{-1}c^{-1})^{m}) \\
&=&(s_{a^{0},b})(s_{a^{1},c})(s_{a^{2},b})(s_{a^{3},c})...(s_{a^{2m-2},b})(s_{a^{2m-1},c})(s_{a^{2m-1},b}^{-1})(s_{a^{2m-2},c}^{-1})...(s_{a^{1},b}^{-1})(s_{a^{0}.c}^{-1})
\\
&=&\Pi _{i=0}^{m-1}(s_{a^{2i},b})(s_{a^{2i+1},c})\Pi
_{i=m}^{1}(s_{a^{2i-1},b}^{-1})(s_{a^{2i-2},c}^{-1})
\end{eqnarray*}%
implying that $$\Pi _{i=0}^{m-1}(s_{a^{2i},b})(s_{a^{2i+1},c})=\Pi
_{i=0}^{m-1}(s_{a^{2i},c})(s_{a^{2i+1},b}).$$

When $P=2l+1,$ we have%
\begin{eqnarray*}
1 &=&(\Pi _{i=0}^{l-1}(s_{a^{2i},a})(s_{a^{2i+1},c}))(s_{a^{2k},a})(\Pi
_{i=l}^{1}(s_{a^{2i},c}^{-1})(s_{a^{2i-1},a}))(s_{0,c}^{-1}) \\
&=&(\Pi _{i=0}^{l-1}(s_{a^{2i+1},c}))(\Pi
_{i=l}^{1}(s_{a^{2i},c}^{-1}))(s_{0,c}^{-1}),
\end{eqnarray*}%
implying $\Pi _{i=0}^{l-1}(s_{a^{2i+1},c})=\Pi _{i=0}^{l}(s_{a^{2i},c}).$ 
\end{proof}

 According to Lemma \ref{shift}, we have $$a^i (s_{a^{m},b}) a^{-i} = s_{a^{i+m},b},$$ 
$$a^i (s_{a^{m},c}) a^{-i} = s_{a^{i+m},c}$$ for any integers $i,m$. The conjugating by $a^i$ of the relators in Lemma \ref{relator} gives new relators for $\ker f$. The exponents of $a$ in the new identities are simply increased by $i$.  We call this phenomenon a degree-shifting.

In the following, we will use the notation $%
S_{i,c}=[s_{a^{i},c}],S_{i,b}=[s_{a^{i},b}]$ to denote the corresponding homology class in $ H_{1}(\mathrm{Art}_{MNP}^{\prime })$.




\begin{proof}[Proof of Theorem \protect\ref{th2}]
Without loss of generality, we assume that $M,P$ are odd. By Lemma \ref%
{relator} and the degree-shifting, we have the following 
\begin{eqnarray}
\label{equb}\sum_{i=1}^{M}(-1)^{i}S_{i,b} &=&0, \\
\label{equc}\sum_{i=1}^{P}(-1)^{i}S_{i,c} &=&0, \\
\label{equbc}\sum_{i=1}^{N}(-1)^{i}S_{i,b} &=&\sum_{i=1}^{N}(-1)^{i}S_{i,c}.
\end{eqnarray}%
Applying a shift of degree $1$ to equation (\ref{equb}), we have 
\begin{eqnarray*}
   \sum_{i=1}^{M}(-1)^{i}S_{i+1,b} = -S_{2,b}+S_{3,b}-S_{4,b}+\cdots -S_{M+1,b} &=&0
\end{eqnarray*}
Now add this to (\ref{equb}), we have $S_{1,b}=-S_{1+M,b}$, and  $S_{i,b}=-S_{i+M,b}$ by degree-shifting. The same method applying to (\ref{equc})  shows that $S_{i+P,c}=-S_{i,c}$.  Suppose that $N=k_{1}M+m_{1},N=k_{2}P+p_{1}$ for integers $k_{1},k_{2} \geq 0$ and $%
0\leq m_{1}<M,0\leq p_{1}<P.$ Since $\gcd (M,N)= \gcd (P,N)=1,$ we know
that $m_{1},p_{1}\neq 0.$ Equation (\ref{equbc}) now reduces to:
\begin{equation}\label{equbc-mp}
\sum_{i=1}^{m_{1}}(-1)^{i}S_{i,b}=\sum_{i=1}^{p_{1}}(-1)^{i}S_{i,c}.
\end{equation}%
We further have 
\begin{equation}\label{equbc-new}
\sum_{i=1}^{m_{1}}(-1)^{i}S_{i,b}=(-1)^{k}%
\sum_{i=1}^{m_{1}}(-1)^{i}S_{i+kM,b}=(-1)^{k}%
\sum_{i=1}^{p_{1}}(-1)^{i}S_{i+kM,c}
\end{equation}%
for any $k\in \mathbb{Z}$, where the first identity follows from the fact that 
$S_{i+M,b}=-S_{i,b}$ for any $i$ and the second identity follows from a
shifting of (\ref{equbc-mp}) of degree $kM$. Since $\gcd (M,P)=1,$ for any integer $j$
there exist $k,l$ such that $j=kM+lP$. Then we have that 
\begin{eqnarray}
\label{eq1}
\sum_{i=1}^{p_{1}}(-1)^{i}S_{i,c} &\overset{(\ref{equbc-mp})}{=}&\sum_{i=1}^{m_{1}}(-1)^{i}S_{i,b} 
 \nonumber,  \\
&\overset{(\ref{equbc-new})}{=}&(-1)^{k}\sum_{i=1}^{p_{1}}(-1)^{i}S_{i+kM,c} 
\nonumber \\
&=&(-1)^{k}\sum_{i=1}^{p_{1}}(-1)^{i}S_{i+j-lP,c} 
\nonumber \\
&=&(-1)^{k-l}\sum_{i=1}^{p_{1}}(-1)^{i}S_{i+j,c} 
\nonumber \\
&=&(-1)^{j}\sum_{i=1}^{p_{1}}(-1)^{i}S_{i+j,c},
\end{eqnarray}%
where the last identity uses that $(-1)^{k-l}=(-1)^{j}$ (in fact, since both 
$M,P$ are odd and $j=kM+lP$, one checks that $j$ is even if and only $k-l$
is even). If $p_{1}=1$, since (6) holds for any $j$, we have 
\begin{equation*}
\sum_{i=1}^{m_{1}}(-1)^{i}S_{i,b}=(-1)^{j}(-1)S_{j+1,c}=(-1)^{j+1}(-1)S_{j+2,c},
\end{equation*}%
which implies 
\begin{equation}\label{equ-j}
S_{j+1,c}=-S_{j+2,c} \text{ for any } j.
\end{equation}%
Since $%
\sum_{i=1}^{P}(-1)^{i}S_{i,c}=0$ and $P$ is odd, we have $S_{j,c}=0$ for any 
$j.$ So let us assume that $p_{1}> 1$ and write $P=k'p_{1}+p_{2}$ for $k'\geq 0,0\leq
p_{2}<p_{1}.$ Since $\gcd (p_{1},P)=\gcd (N,P)=1,$ we know that $p_{2}\neq 0.
$ Then we have%
\begin{eqnarray*}
0 &=&\sum_{i=1}^{P}(-1)^{i}S_{i,c} \\
&=&\sum_{i=1}^{k^{\prime
}p_{1}}(-1)^{i}S_{i,c}+\sum_{i=k'p_1+1}^{k'p_1+p_{2}}(-1)^{i}S_{i,c}
\\
&=&\sum_{i=1}^{p_{1}}(-1)^{i}S_{i,c}+%
\sum_{i=p_{1}+1}^{2p_{1}}(-1)^{i}S_{i,c}+\cdots +\sum_{i=k^{\prime
}p_{1}-p_{1}+1}^{k^{\prime
}p_{1}}(-1)^{i}S_{i,c}+ \sum_{i=k'p_1+1}^{k'p_1+p_{2}}(-1)^{i}S_{i,c}
\\
& 
= &\sum_{i=1}^{p_{1}}(-1)^{i}S_{i,c}+%
\sum_{i=1}^{p_{1}}(-1)^{i+p_{1}}S_{i+p_{1},c}+\cdots
+\sum_{i=1}^{p_{1}}(-1)^{i+k^{\prime }p_{1}-p_{1}}S_{i+k^{\prime
}p_{1}-p_{1},c}+\sum_{i=1}^{p_{2}}(-1)^{k'p_1+i}S_{k'p_1+i,c} \\
&\overset{(\ref{eq1})}{=}&\sum_{i=1}^{p_{1}}(-1)^{i}S_{i,c}+(-1)^{2p_{1}}%
\sum_{i=1}^{p_{1}}(-1)^{i}S_{i,c}+\cdots +(-1)^{2(k^{\prime
}p_{1}-p_{1})}\sum_{i=1}^{p_{1}}(-1)^{i}S_{i,c}+%
(-1)^{2k'p_1}\sum_{i=1}^{p_{2}}(-1)^{i}S_{i,c} \\
&=&k^{\prime
}\sum_{i=1}^{p_{1}}(-1)^{i}S_{i,c}+\sum_{i=1}^{p_{2}}(-1)^{i}S_{i,c} \\
&\overset{(\ref{equbc-mp})}{=}&k^{\prime
}\sum_{i=1}^{m_{1}}(-1)^{i}S_{i,b}+\sum_{i=1}^{p_{2}}(-1)^{i}S_{i,c}.
\end{eqnarray*}%
This means: 
\begin{equation}
\sum_{i=1}^{p_{2}}(-1)^{i}S_{i,c}=-k^{\prime
}\sum_{i=1}^{m_{1}}(-1)^{i}S_{i,b}.  \label{equ3}
\end{equation}%
Applying a degree shifting of $kM$ to the identity (\ref{equ3}), we have 
\begin{equation*}
\sum_{i=1}^{p_{2}}(-1)^{i}S_{i,c}=(-1)^{j}\sum_{i=1}^{p_{2}}(-1)^{i}S_{i+j,c}
\end{equation*}%
for any $j\in \mathbb{Z}$ by exactly the same proof of the identities in (\ref{eq1}). When $%
p_{2}=1,$ we have 
\begin{equation*}
-k^{\prime
}\sum_{i=1}^{m_{1}}(-1)^{i}S_{i,b}=(-1)^{j}(-1)S_{j+1,c}=(-1)^{j+1}(-1)S_{j+2,c},
\end{equation*}%
which agian implies $S_{j,c}=-S_{j+1,c}$ for any $j.$ Since $%
\sum_{i=1}^{P}(-1)^{i}S_{i,c}=0$ and $P$ is odd, we have $S_{j,c}=0$ for any 
$j.$ When $p_{2}> 1,$ we write $p_{1}=k^{\prime \prime }p_{2}+p_{3}$ for $k''\geq0,
0\leq p_{3}<p_{2}.$ Since  $\gcd
(p_{1},P)=1$ and we are just running the Euclidean algorithm for the pair $(P,p_1)$,  we know that $\gcd (p_{1},p_{2})=\gcd (p_{2},p_{3})=1$ and $p_{3}\neq 0.$ Then just as before, we
have 
\begin{eqnarray*}
&&\sum_{i=1}^{m_{1}}(-1)^{i}S_{i,b} \\
&=&\sum_{i=1}^{p_{1}}(-1)^{i}S_{i,c} \\
&=&\sum_{i=1}^{k^{\prime \prime
}p_{2}}(-1)^{i}S_{i,c}+\sum_{i=k''p_2+1}^{k''p_2+p_{3}}(-1)^{i}S_{i,c} \\
&=&\sum_{i=1}^{p_{2}}(-1)^{i}S_{i,c}+%
\sum_{i=p_{2}+1}^{2p_{2}}(-1)^{i}S_{i,c}+\cdots +\sum_{i=k^{\prime
}p_{2}-p_{2}+1}^{k^{\prime \prime
}p_{2}}(-1)^{i}S_{i,c}+\sum_{i=k''p_2+1}^{k''p_2+p_{3}}(-1)^{i}S_{i,c} \\
&=&\sum_{i=1}^{p_{2}}(-1)^{i}S_{i,c}+%
\sum_{i=1}^{p_{2}}(-1)^{i+p_{2}}S_{i+p_{2},c}+\cdots
+\sum_{i=1}^{p_{2}}(-1)^{i+k^{\prime \prime }p_{2}-p_{2}}S_{i+k^{\prime
\prime }p_{2}-p_{2},c}+\sum_{i=1}^{p_{3}}(-1)^{k''p_2+i}S_{k''p_2+i,c} \\
&=&\sum_{i=1}^{p_{2}}(-1)^{i}S_{i,c}+(-1)^{2p_{2}}%
\sum_{i=1}^{p_{2}}(-1)^{i}S_{i,c}+\cdots +(-1)^{2(k^{\prime
}p_{2}-p_{2})}\sum_{i=1}^{p_{2}}(-1)^{i}S_{i,c}+%
(-1)^{2k''p_2}\sum_{i=1}^{p_{3}}(-1)^{i}S_{i,c} \\
&=&k^{\prime \prime
}\sum_{i=1}^{p_{2}}(-1)^{i}S_{i,c}+\sum_{i=1}^{p_{3}}(-1)^{i}S_{i,c} \\
&=&-k^{\prime }k^{\prime \prime
}\sum_{i=1}^{m_{1}}(-1)^{i}S_{i,b}+\sum_{i=1}^{p_{3}}(-1)^{i}S_{i,c}.
\end{eqnarray*}%
Hence
\begin{equation}
\sum_{i=1}^{p_{3}}(-1)^{i}S_{i,c}= (1+k^{%
\prime }k^{\prime \prime })\sum_{i=1}^{m_{1}}(-1)^{i}S_{i,b}.\label{eq4}
\end{equation}%
Note that the right-hand side of the equality (\ref{eq4}) depends only on $%
S_{i,b}$'s. Therefore, we can continue to do degree shiftings to the
equality (\ref{eq4}).  Since $p_1$ and $P$ are coprime, the Euclidean algorithm will eventually arrive at some $p_d= \gcd(p_1,P)=1$. Hence  
$S_{i,c}$ must equal to $0$ for any $i$. A similar argument proves that $%
S_{i,b}=0$ for any $i$. Since the first homology group of $A_{M,N,P}^{\prime
}$ is generated by $S_{i,c},S_{i,b},i=1,2,\ldots ,$ this proves that $%
A_{M,N,P}^{\prime }$ is perfect.
\end{proof}

\subsection{Coherence of triangle Artin groups}

Recall that a group $G$ is coherent if every finitely generated subgroup $H<G
$ is finitely presented.

\begin{corollary}
\label{coherent}Let $\mathrm{Art}_{MNP}$ be a two-dimensional triangle Artin
group such that $M,N,P$ are pairwise coprime, then its commutator subgroup
is not finitely presented. In particular, $\mathrm{Art}_{MNP}$ is not
coherent.
\end{corollary}

\begin{proof}
By a theorem of Bieri \cite[Theorem B]{Bi76}, if its commutator subgroup is
finitely presented, then it must be free. But this contradicts Theorem \ref%
{th2}.
\end{proof}

\begin{remark}
The result in the corollary is known by slightly different methods. In fact,
there is a complete classification of coherent Artin groups, see \cite[%
Section 3]{Go04} and \cite{Wi13}.
\end{remark}

\begin{proposition}

\label{nonfree}Let $f:\mathrm{Art}_{MNP}\rightarrow \mathbb{Z}$ be the
degree map. Then $\ker f$ is not free.
\end{proposition}
\begin{proof}

Suppose that $\ker f$ is free. Feighn and Handel \cite{fh} prove that a
free-by-cyclic group is coherent. Since $f$ is surjective with its kernel
free, this would imply that the group $\mathrm{Art}_{MNP}$ is coherent. 
However, Gordon \cite[Lemma 2.3]{Go04}  shows that a triangle
Artin group of infinite type is incoherent. From Gordon \cite[Lemma 2.4]
{Go04} and Wise \cite{Wi13}, we know that the Artin groups $%
A_{2,3,3},A_{2,3,4},A_{2,3,5}$ are incoherent. Therefore, the only possible coherent
triangle Artin groups are $A_{2,2,n}=\mathbb{Z}\times A_{I_{n}},$ a product
of $\mathbb{Z}$ and the dihedral Artin group $A_{I_{n}}.$ However, it is
clear that $\ker f$ is not free when $\mathrm{Art}_{MNP}=A_{2,2,n}$ by
considering the cohomological dimension.
\end{proof}

\end{document}